\documentclass[a4paper, 11pt, reqno]{amsart}
\usepackage{amsmath}
\usepackage{amscd}
\usepackage{amssymb,enumerate,mathrsfs}
\usepackage{mathtools}
\usepackage[utf8]{inputenc}
\usepackage[english]{babel}
\usepackage{mathrsfs}
\usepackage{amsthm}
\usepackage{amsfonts}
\usepackage{xfrac}
\usepackage{tikz-cd}
\usepackage{hyperref}
\usepackage{theoremref}
\usepackage{imakeidx}
\usepackage{enumitem}
\usepackage{calligra,mathrsfs}
\usepackage{csquotes}
\usepackage{xcolor}
\usepackage{dsfont}
\usepackage[mathscr]{eucal}
\usepackage{comment}
\usepackage[textsize=tiny]{todonotes}

\theoremstyle{definition}
\newtheorem{definition}{Definition}[section]
\newtheorem{example}[definition]{Example}

\theoremstyle{plain}
\newtheorem{theorem}[definition]{Theorem}
\newtheorem{proposition}[definition]{Proposition}
\newtheorem{corollary}[definition]{Corollary}

\newtheorem{lemma}[definition]{Lemma}

\theoremstyle{remark}
\newtheorem{remark}[definition]{Remark}

\DeclareMathOperator{\Hom}{Hom}

\DeclareMathOperator{\Spec}{Spec}

\newcommand{\Z}{\mathbb{Z}}
\newcommand{\A}{\mathbb{A}}

\newcommand{\Mod}{\mathsf{Mod}}

\newcommand{\Coh}{\mathsf{Coh}}

\renewcommand{\tilde}{\widetilde}
\newcommand{\Dqc}{\mathrm{D}_{\sf{qc}}}

\renewcommand{\sf}{\mathsf}

\usepackage{import}
\usepackage[T1]{fontenc}

\newcommand{\qc}{\mathrm{QCoh}}
\newcommand{\Dsqc}{\Ds_{\qc}}
\newcommand{\D}{\mathrm{D}}
\newcommand{\Ds}{\mathcal{D}}
\renewcommand{\H}{\mathrm{H}}

\renewcommand{\mathbb}{\mathbf}

\newcommand{\Orb}{\mathcal{O}}
\DeclareMathOperator{\Map}{Map}
\newcommand{\gms}{\textnormal{gms}}

\newcommand{\LDER}{L}
\newcommand{\ltensor}{\otimes^{\LDER}}

\newcommand{\RDER}{R}

\newcommand{\coh}{\mathrm{Coh}}
\renewcommand{\Mod}{\mathrm{Mod}}
\renewcommand{\Dqc}{\D_{\qc}}

\newcommand{\holim}[1]{\underset{#1}{\mathrm{holim}}\,}
\newcommand{\cl}{\mathrm{cl}}
\newcommand{\Ani}{\mathsf{Ani}}
\newcommand{\Aff}{\mathsf{Aff}}
\newcommand{\STK}{\mathsf{Stk}}
\newcommand{\SHV}{\mathsf{Shv}}
\newcommand{\AniRing}{\sf {AniRing}}
\newcommand{\Ring}{\mathsf{Ring}}
\newcommand{\tCat}{\mathsf{Cat}}
\newcommand{\Cat}{\mathsf{Cat}_{\infty}}
\newcommand{\bCat}{\widehat{\mathsf{Cat}}_{\infty}}
\newcommand{\LPr}{\mathsf{Pr}^{\LDER}}

\numberwithin{equation}{section}
\title[Finiteness theorems]{Finiteness theorems for pseudo-coherent
  complexes on algebraic stacks}
\author{Jack Hall}
\email{jack.hall@unimelb.edu.au}
\author{Oliver Li}
\email{oliver.li@unimelb.edu.au}
\begin{document}
\begin{abstract}
  Let $f\colon X \to Y$ be a proper and tame morphism of algebraic
  stacks, where $X$ and $Y$ are locally of finite type over an algebraic stack $S$. We prove that $R f_*$ sends complexes that are
  pseudo-coherent relative to $S$ to pseudo-coherent complexes relative to $S$. In the scheme case, this resolves a conjecture of Illusie from SGA6. We
  also prove related and new results in the non-tame setting (e.g., infinite
  stabilizers). Our methods use derived algebraic
  geometry and also give new proofs of classical
  statements for schemes due to Kiehl. Along the way, we extend some foundational results for quasi-coherent sheaves on algebraic stacks to the derived setting.
\end{abstract}
\maketitle
\section{Introduction}
Let $f \colon X \to Y$ be a proper morphism of noetherian schemes. A
fundamental finiteness result of Grothendieck \cite[Theorem 3.2.1]{EGAIII1} is
that if $M$ is a coherent $\Orb_X$-module, then the derived
pushforwards $\RDER^if_*M$ are coherent $\Orb_Y$-modules for all
$i\geq 0$. 

In the non-noetherian setting, Illusie conjectured \cite[Conjecture 2.1]{illusie_sga6} a derived reformulation of Grothendieck's result: if $f \colon X \to Y$ is a proper morphism of schemes, where $X$ and $Y$ are locally of finite type over a scheme $S$, then 
$\RDER f_* \colon \D(X) \to \D(Y)$ preserves pseudo-coherent
complexes relative to $S$ (see
\S\ref{S:pseudocoherence} for definitions). This was proved by Illusie in the projective case \cite[Th\'eor\`eme 2.2]{illusie_sga6} and Kiehl in the proper case when $Y=S$ \cite[Theorem 2.9']{Kiehl1972}. This formulation recovers
Grothendieck's result in the noetherian case: if $U \to V$ is a finite type
morphism of noetherian schemes, then a bounded above complex of $\Orb_U$-modules is pseudo-coherent relative to $V$ if and only if it has $\Orb_U$-coherent
cohomology sheaves (Example \ref{E:noetherian-smooth-pcoh-relative}).  Our first result is the following.
\begin{theorem}\label{thm:tame-firstprime}
  Let $f \colon X \to Y$ be a proper and tame (e.g., representable) morphism of algebraic
  stacks, where $X$ and $Y$ are locally of finite type over an algebraic stack
  $S$. Then $\RDER f_*$ preserves pseudo-coherent complexes relative to
  $S$.
\end{theorem}
Note that this is new even for schemes, and
completely settles Illusie's conjecture.
A useful corollary, generalizing \cite[Example 2.2(a)]{LNquasiperfect}
to tame stacks, is the following.
\begin{corollary}\label{cor:tame-perfect}
  Let $f \colon X \to Y$ be a proper, tame, and perfect morphism of
  algebraic stacks. If $P$ is perfect, then $\RDER f_*(P)$ is perfect.
\end{corollary}
In Lemma \ref{lem:commutative-algebra-conditions}, we show that results like Theorem \ref{thm:tame-firstprime} imply, among other things, Corollary \ref{cor:tame-perfect}. Since a proper, concentrated morphism with affine relative stabilizers
is necessarily tame, Theorem \ref{thm:tame-firstprime} and Corollary
\ref{cor:tame-perfect} partially answer a question of Bergh--Schn\"urer \cite[Remark 3.7]{ConservDesc}. We also prove a version for good moduli spaces. 
\begin{theorem}\label{thm:gms}
    Let $f \colon X \to Y$ be a morphism of algebraic stacks with affine diagonal, where $X$ and $Y$ are locally of finite type over an algebraic stack $S$. Assume that $f$ admits a factorization $X \xrightarrow{\pi} X_{\gms} \xrightarrow{f_{\gms}} Y$, where $\pi$ is a relative good moduli space and $f_{\gms}$ is proper and representable. If $Y$ satisfies (FC) (e.g., equicharacteristic), or if $f$ satisfies (PC) or (N) (e.g., tame), then $\RDER f_*$ preserves pseudo-coherent complexes relative to $S$.
\end{theorem}
The conditions (FC), (PC) and (N) are from \cite[\S7]{MR5068605} and are needed for the noetherian approximation used in the proof.

Even in the noetherian situation, direct extensions of Theorem
\ref{thm:tame-firstprime} to the non-tame situation are impossible
(Example \ref{ex:counterexample-non-tame-pushforward}). Nonetheless,
our methods are flexible enough to prove an appropriate generalization
to the non-tame setting.
\begin{theorem}\label{thm:firstprime}
  Let $f\colon X\to Y$ be a proper morphism with finite diagonal of
  algebraic stacks, where $X$ and $Y$ are locally of finite type over an algebraic stack $S$. If
  $P \in \Dqc(X)$ is $Y$-compact, then 
  $\RDER f_*(P \ltensor_{\Orb_X} -)$ preserves pseudo-coherent complexes relative to $S$.
\end{theorem}
Roughly speaking, $Y$-compactness (Definition \ref{D:concentrated}) is a relative version of compactness
in the unbounded derived category that is naturally amenable to base change. If $f \colon X \to Y$ is tame, then
every perfect complex on $X$ is $Y$-compact (Example
\ref{ex:concentrated-rel-compact}), and so Theorem
\ref{thm:firstprime} immediately implies Theorem
\ref{thm:tame-firstprime} (with
$P=\Orb_X$).
\subsection{Methods}
Kiehl's argument in \cite{Kiehl1972} uses noetherian approximation. To
circumvent the failure of base-change for non-flat morphisms, he uses cohomological descent. A na\"ive
attempt to adapt the argument to the setting of algebraic stacks runs
into a host of difficulties, however. These stem from the unboundedness of the
normalised \v{C}ech complex associated to a smooth or \'etale
cover. Thus, the argument does not appear to extend even to
classifying stacks.
    
If one instead uses derived algebraic
geometry to get around the base-change issues, the
difficulties disappear. In fact, Theorems \ref{thm:gms} and \ref{thm:firstprime} are a
consequence of a more general result in derived algebraic geometry
(Theorem \ref{thm:noetherian-base-changes-derived}), which, in
essence, is that concentrated and finite type
morphisms of noetherian (derived) algebraic stacks that are {\em
  cohomologically proper} remain so after base change. Theorems \ref{thm:gms} and \ref{thm:firstprime} then follow from noetherian approximation \cite{rydh2023absolutenoetherianapproximationalgebraic} and the following underived consequence of Theorem \ref{thm:noetherian-base-changes-derived}.
\begin{theorem}\label{thm:mainprime}
  Let $f_0 \colon X_0 \to Y_0$ be a morphism of algebraic stacks,
  where $X_0$ and $Y_0$ are of finite type over a noetherian algebraic
  stack $S_0$. Assume that for every coherent $\Orb_{X_0}$-module
  $M_0$, $\RDER^if_{0,*}M_0$ is a coherent $\Orb_{Y_0}$-module for all
  $i\geq 0$. Let $a \colon S \to S_0$ be a morphism and let
  $f \colon X \to Y$ be the base change of $f_0$ along $a$ and
  $\alpha \colon X \to X_0$ the projection. If $P_0$ is $Y_0$-compact
  and $Y_0$ has quasi-affine diagonal, then
  $\RDER f_*(\LDER\alpha^*P_0 \ltensor_{\Orb_X} - )$ preserves
  complexes that are pseudo-coherent relative to $S$.
\end{theorem}
\subsection{Related results}
If $f$ is pseudo-coherent, representable, and $Y=S$, then Theorem
\ref{thm:tame-firstprime} follows from \cite[Theorem 5.6.0.2]{SAG}. If
$f$ is flat, finitely presented, representable, and $Y=S$, then
Theorem \ref{thm:tame-firstprime} appears in
\cite[\href{https://stacks.math.columbia.edu/tag/0CTN}{Tag
  0CTN}]{stax}. In the preprint
\cite{2012properlocalcompleteintersection}, To\"en proves Theorem
\ref{cor:tame-perfect} in the special case where $f$ is a local
complete intersection morphism of schemes using an idea related to ours, but this extra hypothesis features in a key way. Our main base
change result (Theorem \ref{thm:noetherian-base-changes-derived}) is
similar to \cite[Proposition 2.4.7]{Halpern-Leistner_Preygel_2023},
although neither non-noetherian base changes nor the
non-concentrated case are considered. The idea of reducing the non-noetherian result
to the noetherian result appears in Lurie's thesis
\cite[Proposition 5.5.5]{lurie}, but this method does not seem to
appear in subsequent works. While the connection to Kiehl's theorem is
mentioned in \cite[Remark 5.6.0.3]{SAG}, relative pseudo-coherence is
not discussed---only absolute notions.
\subsection{Organization of the article}
We review pseudo-coherence in
\S\ref{S:pseudocoherence}, almost-perfection in
\S\ref{S:almost-perfection}, and derived algebraic geometry in
Appendices \ref{APP:flatness} and \ref{APP:derived-stacks}. In \S\ref{S:relative-compactness}, we
introduce relative compactness. In the final sections, we prove our
main results in complete generality.
\subsection{Conventions}
For algebraic stacks, we will follow the conventions of
\cite{stax}. For derived algebraic geometry, see Appendices \ref{APP:flatness} and \ref{APP:derived-stacks}, but we roughly follow the conventions
of \cite{BZFN,MR3701352,hansenmann_cllc,khan2026lecturesalgebraicstacks,SAG}. We will
let upper indices and truncations denote cohomological degrees and
lower indices and truncations for homological/homotopy degrees. In particular, for an animated ring
$B$ and $B$-module $M$ we have equivalences of $\pi_0(B)$-modules
$\H^r(M) \simeq \H_{-r}(M) \simeq \pi_{-r}(M)$ and $B$-modules
$\tau^{\leq r}M \simeq \tau_{\geq -r}M$. This convention clashes with
\cite{stax}, where $\tau_{\leq r}$ denotes what we are referring to as
$\tau^{\leq r}$.
\subsection{Acknowledgements}
Both authors were partially supported by the Australian Research
Council DP210103397 and FT210100405.  The first author would like to
thank Siddharth Mathur and Lior Yanovski for some helpful
comments. The second named author would like to thank Marcel Dang,
Dougal Davis, Tianqi Feng, Caleb Ji, Adam Monteleone, Fei Peng, Lior
Yanovski, and Ivan Zelich for several helpful discussions. AI was only
used for proofreading, mathematical feedback, and literature searches
on the completed article.
\section{Classical case}\label{S:pseudocoherence}
Let $B$ be a (classical) ring. Let $\Mod(B)$ denote the abelian
category of $B$-modules and $\D(B)$ its derived category. Recall that
a \emph{perfect complex} of $B$-modules is an object in the derived
category $\D(B)$ obtained from $B$, considered as a complex in degree
zero, by applying finitely many shifts, cones, direct sums and
retracts. Any such complex can be represented by a bounded complex of
finitely generated projective $B$-modules (i.e., a \emph{strictly} perfect
complex).
    
Now let $M$ be a complex of $B$-modules. Let $r\in \Z$. We say
that $M$ is {\em $r$-pseudo-coherent} if there is a morphism
$\alpha \colon P \to M$, where $P$ is perfect, and $\H^i(\alpha)$ is
bijective for all $i>r$ and $\H^r(\alpha)$ is surjective (i.e.,
$\alpha$ is $r$-connected). We say that $M$ is \emph{pseudo-coherent}
if it is $r$-pseudo-coherent for all $r$.
\begin{example}\label{E:noetherian-pcoh}
  If $B$ is noetherian, then $M$ is $r$-pseudo-coherent if and only if
  $\tau^{\geq r}M$ is bounded above with finitely generated cohomology
  \cite[Lemma III.12.3]{Hart}. In this case, we denote by
  $\D^-_{\coh}(B)$ the subcategory of $\D(B)$ consisting of
  pseudo-coherent objects. The standard $t$-structure of $\D(B)$
  restricts to $\D^-_{\coh}(B)$ and
  $\D^-_{\coh}(B)^\heartsuit \simeq \coh(B)$, the abelian category of
  finitely generated $B$-modules.
\end{example}
\begin{example}\label{E:modules-pco}
  If $\tau^{>s}M \simeq 0$, then $M$ is $s$-pseudo-coherent
  if and only if $\H^s(M)$ is a finitely generated
  $B$-module. Similarly, $M$ is $(s-1)$-pseudo-coherent if and only if
  $\H^s(M)$ is a finitely presented $B$-module and $\H^{s-1}(M)$ is a
  finitely generated $B$-module.
\end{example}
There is also a relative version: let $A \to B$ be a map of rings. Let
$M$ be a complex of $B$-modules. Let $r\in \Z$. We say that $M$ is
$r$-pseudo-coherent \emph{relative to $A$} if there is a factorization
$A \to A[x_1,\dots,x_n] \to B$, with the second map
surjective, such that $M$ is $r$-pseudo-coherent as a complex of
$A[x_1,\dots,x_n]$-modules. As before, we say that $M$ is
pseudo-coherent \emph{relative to $A$} if $M$ is $r$-pseudo-coherent
relative to $A$ for all $r$. We also say that $M$ is \emph{perfect
  relative to $A$} if it is pseudo-coherent and of finite
tor-dimension relative to $A$
\cite[\href{https://stacks.math.columbia.edu/tag/0685}{Tag
  0685}]{stax}. We say that $B$ is a \emph{pseudo-coherent} or
\emph{perfect} $A$-algebra if $B$, viewed as a module over itself, is
pseudo-coherent or perfect relative to $A$, respectively.
\begin{example}
  Let $R$ be a ring. Set
  \[
    A = R[t, x_1, x_2,\ldots], B = A/(tx_1, tx_2,\ldots), C = B/tB = A/tA.
  \]
  Then there are surjections
  $A \twoheadrightarrow B \twoheadrightarrow C$. Note that the
  $B$-module $_BC$ is $A$-perfect and thus $A$-pseudo-coherent. Also,
  $_BC$ is finitely presented, and thus $C$ is $(-1)$-pseudo-coherent
  relative to $B$, but is not $(-2)$-pseudo-coherent relative to
  $B$. Also, $_AB$, considered as both an $A$-module and an
  $A$-algebra, is finitely generated but not finitely presented, and
  so is not pseudo-coherent. In particular, relative pseudo-coherence
  does not imply absolute pseudo-coherence.
\end{example}
\begin{example}\label{E:noetherian-smooth-pcoh-relative}
  If $A$ and $B$ are noetherian, then $B$ is a pseudo-coherent
  $A$-algebra if and only if it is a finitely generated
  $A$-algebra. Moreover for complexes of $B$-modules, pseudo-coherence
  and pseudo-coherence relative to $A$ are equivalent. More generally,
  if $B$ is a pseudo-coherent $A$-algebra, then pseudo-coherence and
  pseudo-coherence relative to $A$ are equivalent
  \cite[\href{https://stacks.math.columbia.edu/tag/064Z}{Tag
    064Z}]{stax}. Similarly, if $B$ is a perfect $A$-algebra, then
  perfection and relative perfection are equivalent.
\end{example}
\begin{example}\label{E:flat-fp-pcoh}
  If $A \to B$ is flat and finitely presented, then $B$ is a
  pseudo-coherent $A$-algebra
  \cite[\href{https://stacks.math.columbia.edu/tag/0695}{Tag
    0695}]{stax}.
\end{example}
The following is a special case of
\cite[\href{https://stacks.math.columbia.edu/tag/0673}{Tag
  0673}]{stax}, which we will revisit in the derived context with
Corollary \ref{corollary:pushforward}.
\begin{example}\label{E:relative-pseudo-push}
  Let $A \to B \to C$ be ring homomorphisms such that $B$ is a
  finitely generated $A$-algebra and $B \to C$ is surjective. Let
  $r\in \Z$. Let $M$ be a complex of $C$-modules. Then $M$ is
  $r$-pseudo-coherent relative to $A$ if and only if $M$ is
  $r$-pseudo-coherent relative to $A$ as a complex of $B$-modules. In
  particular, $M$ is pseudo-coherent relative to $A$ as a $C$-module
  if and only if it is pseudo-coherent relative to $A$ as a
  $B$-module. This is trivial: choose a surjection
  $A[x_1,\dots,x_n] \twoheadrightarrow B \twoheadrightarrow C$. Then
  the complex of $C$-modules $M$ is $r$-pseudo-coherent as a complex
  of $A[x_1,\dots,x_n]$-modules if and only if it is
  $r$-pseudo-coherent relative to $A$ as a complex of $C$-modules if
  and only if it is $r$-pseudo-coherent relative to $A$ as a complex
  of $B$-modules.
\end{example}
Since pseudo-coherence, relative pseudo-coherence, and perfection are
all smooth local 
\cite[\href{https://stacks.math.columbia.edu/tag/0CSM}{Tag
  0CSM}]{stax}, these notions extend naturally to the derived category of quasi-coherent
sheaves on an algebraic stack and their morphisms.
\section{Deriving}\label{S:almost-perfection}
In this section, we describe the derived analog of \S\ref{S:pseudocoherence}. 
Let $B$ be an animated ring.  A $B$-module is \emph{perfect} if it is obtained by applying finitely
many shifts, cones, direct sums and retracts to $B$; equivalently, it
is a compact object of $\Ds(B)$, the stable $\infty$-category of $B$-modules. 
Perfect modules are always \emph{almost
  connective} (i.e. bounded above).

Let $r\in \Z$. A $B$-module $M$ is \emph{$r$-perfect} if there is a
morphism $\alpha \colon P \to M$, where $P$ is perfect, 
$\H^i(\alpha)$ is an isomorphism for $i > -r$, and $\H^{-r}(\alpha)$ is
surjective. We say that $M$ is
\emph{pseudo-coherent} if it is $r$-perfect for all $r$ (some authors
also call this \emph{almost perfect}). Note that $r$-perfect and
pseudo-coherent modules are almost connective as well.
\begin{example}\label{ex:pcoh-aperf-compare-discrete}
  Let $B$ be a discrete animated ring. Then a $B$-module $M$ is
  $r$-perfect if and only if $M$, viewed as an object of
  $\D(\pi_0(B))$, is $(-r)$-pseudo-coherent. In particular, $M$ is
pseudo-coherent if and only if $M$ is a pseudo-coherent object of
  $\D(\pi_0(B))$.
\end{example}
\begin{example}\label{ex:aperf-nilp}
  Let $A \to A'$ be a map of animated rings such that the induced ring
  homomorphism $\pi_0(A) \to \pi_0(A')$ is surjective with nilpotent
  kernel $I$ (e.g., $A'=\pi_0(A)$). Let $M$ be an almost connective
  $A$-module.  Then $M$ is an $r$-perfect (resp.~perfect) $A$-module if and only if
  $M \ltensor_A A'$ is an $r$-perfect (resp.~perfect) $A'$-module \cite[Proposition
  2.7.3.2(a),(d)]{SAG}. Similarly,
  $M \ltensor_A A' \simeq 0$ if and only if $M \simeq 0$.
\end{example}
\begin{example}\label{ex:truncation-perfect}
  Let $B$ be an animated ring and $r\in \Z$. Let $M$ be a
  $B$-module. Then $M$ is $r$-perfect if and only if $\tau^{\geq -r}M$
  is $r$-perfect. The necessity is trivial. For the sufficiency: since
  $\tau^{\geq -r}M$ is $r$-perfect, it is almost connective. Hence,
  there is an $m\geq -r$ such that $\tau^{>m}M \simeq 0$. Then
  \cite[Corollary 2.7.2.2]{SAG} implies that there exists a morphism
  $\alpha\colon P \to \tau^{\geq -r}M$, where $P$ is perfect of
  Tor-amplitude $[-r,m]$, $\H^i(\alpha)$ is an isomorphism for $i>-r$,
  and $\H^{-r}(\alpha)$ is surjective. The obstruction to lifting
  $\alpha$ along $M \to \tau^{\geq -r}M$ belongs to
  $\Hom_B(P,\tau^{<-r}M[1])$. Since $P$ is perfect,
  $\Hom_B(P,\tau^{<-r}M[1]) \simeq \H^1(P^\vee \ltensor_B
  \tau^{<-r}M)$. But $P^\vee$ has Tor-amplitude in $[-m,r]$, so
  $\tau^{>0}(P^\vee \ltensor_B \tau^{<-r}M) \simeq 0$, which gives the
  claim.
\end{example}
\begin{example}\label{ex:noetherian-ap}
  Let $B$ be \emph{noetherian}; that is, $\pi_0(B)$ is a (classical)
  noetherian ring and the $\pi_i(B)$ are
  finitely generated $\pi_0(B)$-modules for all $i>0$. Then a $B$-module
  $M$ is $r$-perfect if and only if it is almost connective and
  $\H^{i}(M)$ is a finitely generated $\pi_0(B)$-module for all
  $i\geq -r$. We let $\Ds_{\coh}^-(B)$ denote the subcategory of
  $\Ds(B)$ spanned by the pseudo-coherent $B$-modules. As in the
  classical case, the standard $t$-structure restricts to
  $\Ds_{\coh}^-(B)$ and there is an equivalence of abelian categories
  $\Ds_{\coh}^-(B)^\heartsuit \simeq \coh(\pi_0(B))$.
\end{example}
There is also a relative version of pseudo-coherence: let $A\to B$ be
a map of animated rings. Let $M$ be a $B$-module. We say $M$ is
{$r$-perfect} \emph{relative to $A$} if there is a factorization
$A \to A[x_1,\dots ,x_n] \to B$, with the second map surjective on
$\pi_0$, such that $M$ is $r$-perfect as a
$A[x_1,\dots,x_n]$-module. We say that $M$ is almost
perfect \emph{relative to $A$} if it is $r$-perfect relative to $A$
for all $r\in \Z$. We say the ring map $A \to B$ is \emph{almost
  finitely presented} if $B$ is almost perfect relative to $A$. These
are stable under (derived) base-change.
\begin{remark}\label{rem:well-defined}
  Note that if $A[x_1,\dots,x_n] \to B$ is some map that is surjective
  on $\pi_0$ such that $M$ is an $r$-perfect
  $A[x_1,\dots,x_n]$-module, then $M$ is an $r$-perfect
  $A[y_1,\dots,y_m]$-module for any $A[y_1,\dots,y_m] \to B$ that is
  surjective on $\pi_0$. This follows from a similar argument to
  \cite[\href{https://stacks.math.columbia.edu/tag/065H}{Tag
    065H}]{stax}.
\end{remark}
\begin{remark}\label{rem:compare-discrete}
  Let $A \to B$ be a map of discrete animated rings. Let
  $r\in \Z$. Then the equivalence
  $\mathrm{Ho}(\Ds(B)) \simeq \D(\pi_0(B))$ sends $r$-perfect modules
  relative to $A$ to $(-r)$-pseudo-coherent complexes
  relative to $A$. This follows from Example
  \ref{ex:pcoh-aperf-compare-discrete}.
\end{remark}
The following example is key for our main results (cf.~Corollary \ref{corollary:pushforward}).
\begin{example}\label{ex:pushforward-closed-rel-perfect}
  Let $A \to B \to C$ be morphisms of animated rings such that
  $\pi_0(A) \to \pi_0(B)$ is of finite type and
  $\pi_0(B) \to \pi_0(C)$ is surjective. Let $M$ be a $C$-module. Then
  $M$ is $r$-perfect relative to $A$ as a $C$-module if and only if it
  is so as a $B$-module. 
  Indeed, choose $b_1$, $\dots$,
  $b_n \in \pi_0(B)$ such that the induced map
  $A[\underline{X}] \to B$, where $\underline{X} = (X_1,\dots, X_n)$
  induces surjections
  $\pi_0(A[\underline{X}]) \simeq \pi_0(A)[\underline{X}]
  \twoheadrightarrow \pi_0(B) \twoheadrightarrow \pi_0(C)$. By Remark
  \ref{rem:well-defined}, the $C$-module $M$ is an $r$-perfect
  $A[\underline{X}]$-module if and only if the $B$-module $M$ is an
  $r$-perfect $A[\underline{X}]$-module. That is, $M$, as a
  $B$-module, is $r$-perfect relative to $A$ if and only if it is so
  as a $C$-module.
\end{example}
\begin{example}\label{ex:noetherian-ap-alg-rel}
  Let $A \to B$ be a finite type homomorphism of noetherian classical
  rings. Then $B$ is an almost perfect $A$-algebra. More generally, if
  $A \to B$ is a map of animated rings, where $A$ and $B$ are
  noetherian and $\pi_0(B)$ is a finite type $\pi_0(A)$-algebra, then
  $A \to B$ is almost perfect. In particular, if $C$ is an animated
  $A$-algebra, then $C \ltensor_A B$ is an almost perfect $C$-algebra.
\end{example}
The following example generalizes Example \ref{ex:aperf-nilp} to the relative setting.
\begin{example}\label{ex:relperf-nilp}
  Let $A \to B$ be a map of animated rings and $r \in \Z$. Let $M$ be
  an almost connective $B$-module. Let $A \to A'$ be a map of animated
  rings such that the induced ring homomorphism
  $\pi_0(A) \to \pi_0(A')$ is surjective with nilpotent kernel
  $I$. Then $M$ is $r$-perfect relative to $A$ if and only if
  $M\ltensor_A A'$ is $r$-perfect relative to $A'$.  
  It is clear that 
  if $M$ is $r$-perfect relative to $A$, then $M \ltensor_A A'$ is
  $r$-perfect relative to $A'$. Conversely, if $M \ltensor_A A'$ is
  $r$-perfect relative to $A'$, then it follows that
  $\pi_0(A)/I\simeq \pi_0(A') \to \pi_0(B\ltensor_A A')\simeq \pi_0(B)
  \otimes_{\pi_0(A)} \pi_0(A') \simeq \pi_0(B)/I\pi_0(B)$ is of finite
  type. Since $I$ is nilpotent, it follows that
  $\pi_0(A) \to \pi_0(B)$ is of finite type
  \cite[\href{https://stacks.math.columbia.edu/tag/0G8U}{Tag
   0G8U}]{stax}. Now choose a presentation $A[x_1,\dots,x_n] \to B$
  that induces a surjection on $\pi_0$. Since $M \ltensor_A A'$ is
  $r$-perfect relative to $A'$, it is an $r$-perfect
  $A'[x_1,\dots,x_n]$-module (Remark \ref{rem:well-defined}). It
  follows from Example \ref{ex:aperf-nilp} that $M$ is an $r$-perfect
  $A[x_1,\dots,x_n]$-module, so $M$ is $r$-perfect relative to $A$.
\end{example}
The following proposition is analogous to Example
\ref{E:noetherian-smooth-pcoh-relative}, and the proof is essentially
the same as \cite[\href{https://stacks.math.columbia.edu/tag/064Z}{Tag
  064Z}]{stax}.
\begin{proposition}\label{proposition: rel-abs-agree}
  Let $A \to B$ be an almost finitely presented map of animated rings and
  $r\in \Z$. Let $M$ be a $B$-module. Then $M$ is $r$-perfect if and
  only if it is $r$-perfect relative to $A$, and similarly for pseudo-coherence.
\end{proposition}
\begin{proof}
  Replacing $A$ with a polynomial ring such that
  $\pi_0(A) \to \pi_0(B)$ is surjective, we may assume $B$ is
  pseudo-coherent as an $A$-module. Now apply \cite[Lemma
  5.6.1.2]{SAG}.
\end{proof}
A consequence is that we may replace $A[x_1,\dots,x_n]$ with any
almost finitely presented $A$-algebra $B'$ in the definition of relative
pseudo-coherence.
\begin{corollary}\label{cor:aperf-relative-abs}
  Let $A \to B$ be a map of animated rings and $r \in \Z$. If
  $M$ is a $B$-module, then the following conditions are equivalent.
  \begin{enumerate}[label = \normalfont(\roman*)]
  \item \label{cor:aperf-relative-abs:def} $M$ is $r$-perfect relative to $A$.
  \item \label{cor:aperf-relative-abs:exist} There is a factorization
    $A \to B' \to B$ with $B'$ an almost finitely presented
    $A$-algebra, such that $\pi_0(B') \to \pi_0(B)$ is surjective and $M$ is $r$-perfect as a $B'$-module.
  \item \label{cor:aperf-relative-abs:every} For every factorization
    $A \to B' \to B$, where $B'$ is an almost finitely presented
    $A$-algebra such that $\pi_0(B') \to \pi_0(B)$ is surjective, $M$
    is $r$-perfect as a {$B'$-module}.
  \end{enumerate}
\end{corollary}
\begin{proof}
 For
 \ref{cor:aperf-relative-abs:def}$\Rightarrow$\ref{cor:aperf-relative-abs:exist}:
 take $B' = A[x_1,\dots ,x_n]$, where $A[x_1,\dots ,x_n]$ is as in
 the definition of relative $r$-perfection. For
 \ref{cor:aperf-relative-abs:exist}$\Rightarrow$\ref{cor:aperf-relative-abs:def}:
 combine Proposition \ref{proposition: rel-abs-agree} with Example
 \ref{ex:pushforward-closed-rel-perfect}.

 The implication
 \ref{cor:aperf-relative-abs:every}$\Rightarrow$\ref{cor:aperf-relative-abs:exist}
 is trivial, so it remains to prove
 \ref{cor:aperf-relative-abs:def}$\Rightarrow$\ref{cor:aperf-relative-abs:every}. Fix
 a factorization $A \to B' \to B$ as in
 \ref{cor:aperf-relative-abs:every}. Now choose $b_1'$, $\dots$,
 $b_n' \in \pi_0(B')$ such that the induced $A$-algebra homomorphism
 $A[x_1,\dots,x_n] \to B'$, where $x_i \mapsto b_i'$ induces a
 surjection $\pi_0(A[x_1,\dots,x_n]) \twoheadrightarrow
 \pi_0(B')$. It follows that the composition
 $\pi_0(A[x_1,\dots,x_n]) \to \pi_0(B') \to \pi_0(B)$ is surjective,
 so $M$ is an $r$-perfect $A[x_1,\dots,x_n]$-module (Remark
 \ref{rem:well-defined}). Then $M$ is $r$-perfect relative to $A$ as
 a $B'$-module and so $M$ is $r$-perfect as a $B'$-module
 (Proposition \ref{proposition: rel-abs-agree}).
\end{proof}
Since perfect, $r$-perfect, pseudo-coherent, and their relative variants are all smooth local on animated rings, these notions naturally extend to derived algebraic stacks and their morphisms. Moreover, these agree with the classical notions discussed in \S\ref{S:pseudocoherence}. 
\section{Relative compactness}\label{S:relative-compactness}
In this section, we will show that a number of results of \cite[\S4.2]{perfect_stax} extend to the derived and relative setting. This is somewhat subtle due to the way we have set up quasi-coherent sheaves. Let $X$ be a quasi-compact and quasi-separated derived algebraic
stack. We say that $P \in \Dsqc(X)$ is \emph{compact} if
$\RDER\Hom_{\Orb_X}(P,-)$ preserves small coproducts (equivalently, small filtered colimits) in $\Dsqc(X)$. Some notation: $\RDER \Gamma(X,N) = \RDER\Hom_{\Orb_X}(\Orb_X,N)$. We begin with the following.
\begin{lemma}\label{lem:strongproj}
  Let $X$ be a quasi-compact and quasi-separated derived stack. Let
  $P \in \Dsqc(X)$ be perfect and $M \in \Dsqc(X)$.
  \begin{enumerate}[label = \normalfont(\roman*)]
  \item\label{lem:strongproj:ideal} If $P$ is compact and $Q$ is perfect, then $P \ltensor_{\Orb_X} Q$ is compact.
  \item\label{lem:strongproj:dual} $P$ is compact if and
   only if $P^\vee$ is compact. In particular, $P$ is compact if and
    only if $\RDER \Gamma(X,P\ltensor_{\Orb_X} - )$ preserves small
    coproducts in $\Dsqc(X)$.
    \item\label{lem:strongproj:fcd} $P$ is compact if and only if there exists $d(P)$ such that for all $F \in \Dsqc(X)^{\heartsuit}$:
    \[
      \tau^{>d(P)}\RDER \Gamma(X,P \ltensor_{\Orb_X} F) \simeq 0.
    \]
  \item \label{lem:strongproj:bound} If $P$ is compact, then there exists a minimal $d(P)\geq 0$ such that for all $n\in \Z$:
   \begin{equation}\label{eqn:truncation_cd}
       \tau^{\geq n}\RDER \Gamma(X,P \ltensor_{\Orb_X} M) \simeq \tau^{\geq n}\RDER\Gamma(X,P \ltensor_{\Orb_X} \tau^{\geq n-d(P)}M).
   \end{equation}
 \end{enumerate}
\end{lemma}
\begin{proof}
Claim \ref{lem:strongproj:ideal} is trivial from adjunction. Claim \ref{lem:strongproj:dual} follows from \ref{lem:strongproj:ideal} and that $P$ is a direct summand of $P \ltensor_{\Orb_X} P^\vee \ltensor_{\Orb_X} P$. We now claim that if the cohomological vanishing condition of \ref{lem:strongproj:fcd} holds, then \eqref{eqn:truncation_cd} does too. Indeed, if $M$ is almost coconnective (i.e., bounded below), then \eqref{eqn:truncation_cd} follows from the hypercohomology spectral sequence \cite{antieau2024spectralsequencesdecalagebeilinson}:
  \[
  E_2^{p,q} = \H^p(X, P \ltensor_{\Orb_X} \mathcal{H}^q(M)) \Rightarrow \H^{p+q}(X, P \ltensor_{\Orb_X} M). 
  \] 
  In general, by left completeness (Lemma \ref{L:right-and-left-completeness}), there is an equivalence:
  \[ 
  M \simeq \holim{r}\tau^{\geq -r}M. 
  \]
  Since $P$ is perfect, it follows that
  \[
  \RDER\Gamma(X,P\ltensor_{\Orb_X} M) \simeq \holim{r} \RDER \Gamma(X,P \ltensor_{\Orb_X} \tau^{\geq -r}M).
  \]
  Let us fix $n$, and $i \geq n$. If $r > d(P)-n$, then 
  \[
  \H^i(X, P\ltensor_{\Orb_X} \tau^{\geq -r}M) \simeq \H^i(X, P\ltensor_{\Orb_X} \tau^{\geq n - d(P)}\tau^{\geq -r}M) \simeq \H^i(X, P\ltensor_{\Orb_X} \tau^{\geq n - d(P)}M).
  \] 
  Thus, by the Milnor exact sequence:
  \[
  0 \to \varprojlim_r\!{}^1 \H^{i-1}(X,P\ltensor_{\Orb_X} \tau^{\geq -r}M) \to \H^i(X,P\ltensor_{\Orb_X} M) \to \varprojlim_r \H^i(X,P\ltensor_{\Orb_X} \tau^{\geq -r}M) \to 0
  \]
  it suffices to show that the groups $\H^i(X, P \ltensor_{\Orb_X}\tau^{\geq -r} M)$ stabilize as $r \to \infty$. Now for any $r$, we have a fiber sequence 
  \[
  \mathcal{H}^{-r}(M)[r] \to \tau^{\geq -r}M \to \tau^{\geq -r+1}M. 
  \] 
  Thus, for $r > d(P) - n$, applying $R\Gamma(X, P\ltensor_{\Orb_X}-)$ and taking the long exact sequence, the vanishing of $\H^{i + r}(X, P \ltensor_{\Orb_X}\mathcal{H}^{-r}(M))$ and $\H^{i +1 + r}(X, P \ltensor_{\Orb_X}\mathcal{H}^{-r}(M))$ implies that 
  \[
  \H^i(X, P \ltensor_{\Orb_X}\tau^{\geq -r} M) \simeq \H^i(X, P \ltensor_{\Orb_X}\tau^{\geq -r+1} M).
  \] 
  We now show that \eqref{eqn:truncation_cd} implies that $P$ is compact, which is \ref{lem:strongproj:fcd}. By Example \ref{ex:bounded-below-colimits}, $\RDER \Gamma(X, P \ltensor_{\Orb_X} \tau^{\geq n- d(P)}(-))$ preserves small coproducts. Since $\Ds(\Z)$ is left complete, it now follows from \eqref{eqn:truncation_cd} that $\RDER \Gamma(X,P \ltensor_{\Orb_X} - )$ preserves small coproducts. 
  
  For \ref{lem:strongproj:bound}: it suffices to show that there exists a $d(P)$ satisfying \eqref{eqn:truncation_cd}. Arguing as in \cite[Lemma 4.5(1)$\Rightarrow$(2)]{perfect_stax}, we can find some $d(P)$ such that $\H^i(X, P\ltensor_{\Orb_X} F) = 0$ for all $F\in \Dsqc(X)^\heartsuit$ and $i > d(P)$. Then we have \ref{lem:strongproj:fcd} and so \eqref{eqn:truncation_cd}.
\end{proof}
\begin{remark}\label{rem:classical-compact}
    Let $X$ be a quasi-compact and quasi-separated derived algebraic stack. By Lemma \ref{lem:strongproj}\ref{lem:strongproj:fcd}, if $P \in \Dsqc(X)$ is perfect, then $P$ is compact if and only if its restriction to $X_{\cl}$ is so. In particular, if $X$ is a derived algebraic space, then $\Orb_X$ is compact \cite[\href{https://stacks.math.columbia.edu/tag/073G}{Tag
      073G}]{stax}. More generally, $X_\cl$ has finite cohomological dimension (e.g., tame, or quasi-compact and quasi-separated with affine stabilizers in characteristic $0$ \cite[Theorem C]{classifyingstacks}) if and only if $\Orb_X$ is compact. Hence, every perfect complex is compact if and only if $X_{\cl}$ has finite cohomological dimension.
\end{remark}
\begin{definition}\label{D:concentrated}
    Let $f \colon X \to S$ be a quasi-compact and quasi-separated morphism
of derived algebraic stacks. Let $P \in \Dsqc(X)$. We say that $P$
is \emph{$S$-compact} if for every
smooth morphism $\Spec A \to S$, the restriction of $P$ to
$X_A = X \times^\RDER_{S} \Spec A$ is a compact object of $\Dsqc(X_A)$. If $\Orb_X$ is $S$-compact, then we say that $f$ is \emph{concentrated}.
\end{definition}
This condition is clearly stable under smooth base change on $S$. In particular if $P$ is perfect, then it is $S$-compact if and only if $P^\vee$ is.
\begin{example}\label{ex:concentrated-rel-compact}
  If $f\colon X \to S$ is concentrated, then every perfect complex on $X$ is $S$-compact. Indeed, $P$ is
  perfect and $\Spec A \to S$ is smooth, so $P_{X_A}$ is perfect on
  $X_A$. But $X_A$ has finite cohomological dimension since $f$ is
  concentrated, so $P_{X_A}$ is compact in $\Dsqc(X_A)$ (Remark \ref{rem:classical-compact}). Thus, $P$ is $S$-compact.
\end{example}
\begin{example}\label{ex:concentrated-morphisms}
  Let $f\colon X \to S$ be a quasi-compact and quasi-separated
  morphism of derived algebraic stacks. If $f$ is tame (e.g.,
  representable) or has relatively affine stabilizers and $S$ has
  characteristic $0$, then $f$ is concentrated. This follows from
  Remark \ref{rem:classical-compact}. More generally, $f$ is
  concentrated if and only if $f_{\cl} \colon X_{\cl} \to S_{\cl}$ is
  concentrated as in \cite{perfect_stax}.
\end{example}
\begin{example}\label{ex:abs-rel-compact}
    Let $f \colon X \to \Spec A$ be a quasi-compact and quasi-separated morphism of derived stacks. If $P\in \Dsqc(X)$, then $P$ is $\Spec A$-compact if and only if it is compact. One direction is trivial. For the other, it suffices to prove that in the cartesian square:
    \[
    \begin{tikzcd}
        X' \arrow[r,"\gamma"] \arrow[d,"f'"'] & X  \arrow[d,"f"]\\ \Spec A' \arrow[r,"g"'] & \Spec A,
    \end{tikzcd}
    \]
    that $\LDER \gamma^*P$ is compact. But Example \ref{ex:base-change-affine} shows that $\RDER \gamma_*$ preserves filtered colimits, so $\LDER \gamma^*$ preserves compact objects. 
\end{example}
The following is a relative version of Lemma \ref{lem:strongproj} (cf.~\cite[Proposition 9.1.5.7]{SAG}).
\begin{proposition}\label{prop:relative-compactness-properties}
  Let $f \colon X \to S$ be a quasi-compact and quasi-separated
  morphism of derived algebraic stacks. Let $P \in \Dsqc(X)$ be perfect and $S$-compact and let $M \in \Dsqc(X)$.
  \begin{enumerate}[label = \normalfont(\roman*)]
  \item \label{prop:relative-compactness-properties:ideal} If $Q \in \Dsqc(X)$ is perfect, then $P\ltensor_{\Orb_X} Q$ is $S$-compact.
  \item \label{prop:relative-compactness-properties:sum} The functor $\RDER f_*(P \ltensor_{\Orb_X} - )$ preserves small coproducts.
  \item \label{prop:relative-compactness-properties:projection} If $N \in \Dsqc(S)$, then there is a natural equivalence:
    \[
      \RDER f_*(P \ltensor_{\Orb_X} M) \ltensor_{\Orb_S} N \simeq
      \RDER f_*(P \ltensor_{\Orb_X} M \ltensor_{\Orb_X} \LDER f^*N).
    \]
  \item \label{prop:relative-compactness-properties:bound} If $S$ is
    quasi-compact, then there exists a minimal $d(P,f)\geq 0$ such that  $\forall\,n\in \Z$:
    \[
      \tau^{\geq n} \RDER f_*(P\ltensor_{\Orb_X} M) \simeq \tau^{\geq n}\RDER f_* (P\ltensor_{\Orb_X}\tau^{\geq n-d(P,f)}M).
    \]
    In particular, if $M$ is almost connective,
    then $\RDER f_*(P \ltensor_{\Orb_X} M)$ is almost connective.
  \item \label{prop:relative-compactness-properties:base-change}
    Consider a fiber diagram of derived algebraic stacks:
    \[
      \begin{tikzcd}
        X' \arrow[r, "\gamma"] \arrow[d, "f'"'] & X \arrow[d, "f"] \\
        S' \arrow[r, "g"'] & S.
      \end{tikzcd}
    \]
    Then $\LDER \gamma^*P$ is $S'$-compact and there is a natural
    equivalence:
    \[
      \LDER g^*\RDER f_*(P\ltensor_{\Orb_X} M) \simeq \RDER f'_*\LDER
      \gamma^*(P\ltensor_{\Orb_X} M).
    \]
    Moreover, if $S$ and $S'$ are quasi-compact, then
    $d(\LDER \gamma^*P,f') \leq d(P,f)$.
  \end{enumerate}
\end{proposition}
\begin{proof} 
  Claim \ref{prop:relative-compactness-properties:ideal} is trivial
  from Lemma \ref{lem:strongproj}\ref{lem:strongproj:ideal}. For
  \ref{prop:relative-compactness-properties:sum}-\ref{prop:relative-compactness-properties:base-change},
  we first consider the case where $S'=\Spec A'$ and $S=\Spec A$. In
  this case, \ref{prop:relative-compactness-properties:sum} and
  \ref{prop:relative-compactness-properties:bound} are just Lemma
  \ref{lem:strongproj}\ref{lem:strongproj:dual} and
  \ref{lem:strongproj:bound}, respectively, with $d(P,f) = d(P)$. For
  \ref{prop:relative-compactness-properties:projection}: adjunction
  gives a natural morphism
  \[
    \phi(N) \colon \RDER\Gamma(X,P\ltensor_{\Orb_X} M) \ltensor_A N \to
    \RDER \Gamma(X,P\ltensor_{\Orb_X} M \ltensor_{\Orb_X} \LDER f^*N)
  \]
  such that $\phi(A)$ is an equivalence. Now let
  $\mathcal{T} \subseteq \Ds(A)$ be the full subcategory spanned by
  those $N$ such that $\phi(N)$ is an equivalence. Then $\mathcal{T}$
  is stable, contains $A$, and is closed under small coproducts
  because $P$ is compact in $\Dsqc(X)$. In particular,
  $\mathcal{T} = \Ds(A)$ and we have the claim. For \ref{prop:relative-compactness-properties:base-change}: \ref{prop:relative-compactness-properties:projection} tells us that there are natural equivalences:
  \begin{align*}
    \RDER\Gamma(X,P\ltensor_{\Orb_X} M) \ltensor_A A' &\simeq \RDER \Gamma(X,P \ltensor_{\Orb_X} M \ltensor_{\Orb_X} \LDER f^* \RDER g_*\Orb_{\Spec A'})\\
                                        &\simeq \RDER \Gamma(X,P \ltensor_{\Orb_X} M \ltensor_{\Orb_X} \RDER \gamma_*\Orb_{X'}) & \mbox{(Example \ref{ex:base-change-affine})}\\
                                        &\simeq \RDER \Gamma(X,\RDER \gamma_*\LDER \gamma^*(P \ltensor_{\Orb_X} M)) & \mbox{(Example \ref{ex:affine-projection-formula})}\\
    &\simeq \RDER \Gamma(X',\LDER \gamma^*P\ltensor_{\Orb_{X'}}\LDER \gamma^*M). 
  \end{align*}
  That $\LDER \gamma^*P$ is $S'$-compact is Example \ref{ex:abs-rel-compact}. By Example \ref{ex:affine-projection-formula}, if $M' \in \Dsqc(X')$, then
  \begin{align*}
    \RDER\Gamma(X',\LDER \gamma^*P\ltensor_{\Orb_{X'}} M') &\simeq \RDER \Gamma(X,\RDER \gamma_*(\LDER \gamma^*P\ltensor_{\Orb_{X'}} M'))\\
    &\simeq \RDER \Gamma(X,P \ltensor_{\Orb_X} \RDER \gamma_*M').
  \end{align*}
  Since $\RDER \gamma_*$ is $t$-exact, $d(\LDER \gamma^*P,f') \leq d(P,f)$.

Next note that \ref{prop:relative-compactness-properties:sum}-\ref{prop:relative-compactness-properties:bound} in general follow from the affine case already considered and the general case of \ref{prop:relative-compactness-properties:base-change}. Thus, it remains to prove \ref{prop:relative-compactness-properties:base-change} in general. Still keeping $S=\Spec A$ but now allowing $S'$ to be arbitrary, we can use that $\Dsqc(S') = \varprojlim_{i\in I'} \Ds(A_i')$ and $\Dsqc(X') = \varprojlim_{i\in I'} \Dsqc(X_i')$, where $X_i' = X'\times^\RDER_{S'} \Spec A_i'$ for some category $I'$. By the affine case already considered, it follows from Lemma \ref{lem:bc-prl-limits} that we have the claim. In general, write $\Dsqc(S) = \varprojlim_{i\in I} \Ds(A_i)$. Then $\Dsqc(X) = \varprojlim_{i\in I} \Dsqc(X_i)$, where $X_i=X\times^\RDER_S \Spec A_i$; $\Dsqc(S') = \varprojlim_{i\in I} \Dsqc(S_i')$, where $S'_i = S'\times_S^\RDER \Spec A_i$; and $\Dsqc(X') = \varprojlim_{i\in I} \Dsqc(X'_i)$, where $X'_i=X'\times_{S'}^\RDER S'_i$. By the case already considered, the result now follows again from Lemma \ref{lem:bc-prl-limits}.
\end{proof}
We finally get to our relationship between compactness and perfection.
\begin{corollary}\label{cor:compact-is-perfect}
Let $f \colon X \to S$ be a quasi-compact and quasi-separated morphism of derived algebraic stacks. Let $P \in \Dsqc(X)$.
\begin{enumerate}[label = \normalfont(\roman*)]
\item\label{cor:compact-is-perfect:restrict} If $Q \in \Dsqc(S)$ is compact and $P$ is $S$-compact, then $P\ltensor_{\Orb_X}\LDER f^*Q \in \Dsqc(X)$ is compact. In particular, if $f$ is concentrated, then $\LDER f^*Q$ is compact.
\item\label{cor:compact-is-perfect:compact} If $P$ is $S$-compact, then $P$ is perfect.
\item\label{cor:compact-is-perfect:relcompact} If $S$ is quasi-separated and $P$ is compact, then $P$ is $S$-compact.
\item\label{cor:compact-is-perfect:fcd} If $S$ is quasi-compact, quasi-separated, and finite cohomological dimension, then $P$ is $S$-compact if and only if it is compact.
\item\label{cor:compact-is-perfect:cl} If $P$ is almost connective, then it is $S$-compact if and only if $P_{X_{\cl}}$ is $S_{\cl}$-compact.
\end{enumerate}
\end{corollary}
\begin{proof}
  We first treat \ref{cor:compact-is-perfect:restrict} when $P$ is
  perfect.  By Proposition \ref{prop:relative-compactness-properties}
  \ref{prop:relative-compactness-properties:sum}, the functor
  $\RDER f_*(P \ltensor_{\Orb_X} - ) \colon \Dsqc(X) \to \Dsqc(S)$
  preserves small coproducts. Hence, its left adjoint
  $P^\vee \ltensor_{\Orb_X} \LDER f^*(-)$ preserves compacts
  \cite[Lemma 5.5.1.4]{HTT}. This applies when $P=\Orb_X$, so if $f$ is
  concentrated, then $\LDER f^*Q$ is compact. For
  \ref{cor:compact-is-perfect:compact}: if $\Spec A \to S$ is smooth,
  then the restriction of $P$ to $X_A = X\times_S^\RDER \Spec A$ is
  compact in $\Dsqc(X_A)$. Since $X_A$ is quasi-compact and
  quasi-separated, there is a smooth surjection
  $p\colon \Spec B \to X_A$, which is quasi-compact, quasi-separated,
  and representable. Hence, $p$ is concentrated (Example
  \ref{ex:concentrated-morphisms}). By the case of
  \ref{cor:compact-is-perfect:restrict} already considered,
  $\LDER p^*P_{X_A} \in \Ds(B)$ is compact, so is perfect. To complete
  \ref{cor:compact-is-perfect:restrict}: we combine with perfect case
  already established with \ref{cor:compact-is-perfect:compact}. For
  \ref{cor:compact-is-perfect:relcompact}: since $S$ is
  quasi-separated, the map $X_A \to X$ above is concentrated, so
  \ref{cor:compact-is-perfect:restrict} implies that $P_{X_A}$ is
  compact on $X_A$. Hence, $P$ is $S$-compact. For
  \ref{cor:compact-is-perfect:fcd}: if $P$ is compact, then it is
  $S$-compact by \ref{cor:compact-is-perfect:relcompact}. Conversely,
  if $P$ is $S$-compact, then it is perfect by
  \ref{cor:compact-is-perfect:compact}. Hence, it suffices to prove
  that
  $\RDER \Gamma(X, P \ltensor_{\Orb_X} - ) \simeq \RDER \Gamma(S ,
  \RDER f_*(P \ltensor_{\Orb_X} - ))$ preserves small coproducts in
  $\Dsqc(X)$ (Lemma 
  \ref{lem:strongproj}\ref{lem:strongproj:dual}). Since $S$ is
  concentrated, $\Orb_S$ is a compact object of $\Dsqc(S)$, so
  $\RDER \Gamma(S,-)$ preserves small coproducts in
  $\Dsqc(S)$. Proposition
  \ref{prop:relative-compactness-properties}\ref{prop:relative-compactness-properties:sum}
  implies that $\RDER f_*(P \ltensor_{\Orb_X} -)$ preserves small
  coproducts in $\Dsqc(X)$, and the result follows. For
  \ref{cor:compact-is-perfect:cl}: we may assume that $S$ and $S_\cl$
  are affine. By Example \ref{ex:abs-rel-compact}, it suffices to
  prove that $P$ is compact if and only if $P_{X_{\cl}}$ is
  compact. If $i\colon X_{\cl} \to X$ denotes the inclusion, then $i$
  is affine and so concentrated (Example
  \ref{ex:concentrated-morphisms}). By
  \ref{cor:compact-is-perfect:restrict}, it follows that if $P$ is
  compact, then $P_{X_\cl}$ is compact. Conversely, if $P_{X_{\cl}}$
  is compact, then \ref{cor:compact-is-perfect:compact} implies that
  $P_{X_{\cl}}$ is perfect. By Example \ref{ex:aperf-nilp}, $P$ is
  perfect. Now apply Remark \ref{rem:classical-compact}.
\end{proof}
\section{Cohomological properness}
Let $f \colon X \to S$ be an almost perfect morphism of derived
locally noetherian algebraic stacks. We say that $f$ is
\emph{cohomologically proper} if
$\RDER^n f_*M=\mathcal{H}^n(\RDER f_*M) \in \Ds_{\coh}^-(S)^\heartsuit$ for all
$M \in \Ds_{\coh}^-(X)^\heartsuit$ and $n\in \Z$. To the authors' knowledge, cohomological
properness was first studied systematically in
\cite{Halpern-Leistner_Preygel_2023}. Also see
\cite{Coherent_complete} for some variations. We begin with the
following lemma.
\begin{lemma}\label{lem:compact-pushforward}
  Let $f \colon X \to S$ be a cohomologically proper morphism of
  derived noetherian algebraic stacks.
  \begin{enumerate}[label = \normalfont(\roman*)]
  \item \label{lem:compact-pushforward:bdd_below} $\RDER f_*$ sends
    $\Ds_{\coh}^+(X)$ to $\Ds_{\coh}^+(S)$.
  \item \label{lem:compact-pushforward:compact} If $P \in \Dsqc(X)$ is
    $S$-compact, then $\RDER f_*(P \ltensor_{\Orb_X} - )$ preserves pseudo-coherents.
  \end{enumerate}
\end{lemma}
\begin{proof}
  Claim \ref{lem:compact-pushforward:bdd_below} follows immediately
  from the hypercohomology spectral sequence
  \cite{antieau2024spectralsequencesdecalagebeilinson}:
  \[
    E_2^{p,q} = \RDER^pf_*\mathcal{H}^q(M) \Rightarrow \RDER^{p+q}f_*M
  \]
  For \ref{lem:compact-pushforward:compact}: if $M \in \Ds^-_{\coh}(X)$, by Example \ref{ex:truncation-perfect}, it suffices to prove that for each $n\in \Z$, $\tau^{\geq n}\RDER f_*(P \ltensor_{\Orb_X} M) \in \Ds^-_{\coh}(S)$. By Proposition
  \ref{prop:relative-compactness-properties}\ref{prop:relative-compactness-properties:bound}
 there is a $d(P,f)$ such that 
  $\tau^{\geq n}\RDER f_*(P \ltensor_{\Orb_X} M) \simeq \tau^{\geq n}
  \RDER f_*(P \ltensor_{\Orb_X} \tau^{\geq n - d(P,f)}M)$. Hence, we
  may replace $M$ by $\tau^{\geq n-d(P,f)}M \in \Ds_{\coh}^b(X)$. Since $P$ is perfect,
  $P \ltensor_{\Orb_X} M \in \Ds_{\coh}^b(X)$. Now apply
  \ref{lem:compact-pushforward:bdd_below}.
\end{proof}
That $P$ is compact is essential in
Lemma \ref{lem:compact-pushforward}\ref{lem:compact-pushforward:compact}
(cf.~\cite[\href{https://stacks.math.columbia.edu/tag/07DC}{Tag
  07DC}]{stax}).
\begin{example}\label{ex:counterexample-non-tame-pushforward}
  Let $k$ be a field of characteristic $2$. Let
  $f\colon X = B(\Z/2) \to S=\Spec k$. It is easy to calculate
  directly that
  $\RDER f_*\Orb_{X} \simeq \bigoplus_{n\leq 0}\Orb_S[n]$. Let
  $M = \bigoplus_{r\geq 0}\Orb_X[r] \simeq \prod_{r\geq 0} \Orb_X[r]
  \in \Ds_{\Coh}^-(X)$ (by Lemma \ref{L:right-and-left-completeness}); then $\RDER f_*$ and $\H^0$ commute with products, so
  \[
    \H^0(\RDER f_*M) \simeq \prod_{r\geq 0} \H^0(\RDER f_*\Orb_X[r]) \simeq \prod_{r \geq 0} \Orb_S \notin \Coh(S).
  \]
  
\end{example}
The following is a derived version of Theorem \ref{thm:mainprime}, the
main result of this section. This is similar to \cite[Proposition
2.4.7]{Halpern-Leistner_Preygel_2023}.
\begin{theorem}\label{thm:noetherian-base-changes-derived}
  Let $f_0 \colon X_0 \to S_0$ be a cohomologically proper morphism of derived  noetherian algebraic stacks. Let $a\colon S \to S_0$ be a morphism. Let
  $f \colon X \to S$ be the derived base change and
  $\alpha \colon X \to X_0$ the projection. Let $P_0 \in \Dsqc(X_0)$ be
  $S_0$-compact. If
  \begin{enumerate}[label = \normalfont(\roman*)]
  \item \label{thm:noetherian-base-changes-derived:qaff-noeth} $S$ is
    noetherian and $a$ is quasi-affine; or
  \item \label{thm:noetherian-base-changes-derived:qaff-diag} $S_0$
    has quasi-affine diagonal;
  \end{enumerate}
  then $\RDER f_*(\LDER\alpha^*P_0 \ltensor_{\Orb_X} - )$ preserves pseudo-coherents. 
\end{theorem}
\begin{proof}
  This is smooth-local on $S$, we may assume that $S=\Spec A$ is quasi-compact. Set
  $P=\LDER \alpha^*P_0$. Let $d=d(P_0,f_0)\geq d(P, f)\geq 0$ be as in
  Proposition
  \ref{prop:relative-compactness-properties}\ref{prop:relative-compactness-properties:bound}. If
  $N\in \Dsqc(X)$ is almost connective, then
  $\RDER f_*(P \ltensor_{\Orb_X} N)$ is almost connective. Combining
  Example \ref{ex:aperf-nilp} with Proposition
  \ref{prop:relative-compactness-properties}\ref{prop:relative-compactness-properties:base-change},
  we may further assume $S$ is classical. 

  For \ref{thm:noetherian-base-changes-derived:qaff-noeth}: it suffices to prove by
  induction on $t$ that
  $\tau^{\geq -t}\RDER f_*(P \ltensor_{\Orb_X} M) \in \Ds_{\coh}^-(S)$
  for all connective $M \in \Ds_{\coh}^-(X)$. By a simple filtration
  argument (noting the existence of the bound $d$), if the inductive hypothesis is true for all
  $M \in \Ds_{\coh}^-(X)^\heartsuit$, then it is true for all
  connective $M$. The base case of the induction is $t=-d-1$. This
 holds because
  \[
    \tau^{\geq d+1}\RDER f_*(P \ltensor_{\Orb_X} M) \simeq \tau^{\geq d+1}\RDER f_*(P \ltensor_{\Orb_X} \tau^{\geq 1}M) = 0.
  \]
  Now assume the result is true for all $k\leq t$. We must
  prove it for $k=t+1$. As noted above, it suffices to assume that
  $M \in \Ds_{\coh}^-(X)^\heartsuit$. Let
  $\alpha_{\cl} \colon X_\cl \to X_{0,\cl}$ be the induced
  quasi-affine morphism between the underlying classical algebraic
  stacks. By Remark \ref{rem:heart-derived} and Example \ref{ex:noetherian-ap}, 
  $\mathcal{H}^0(M)$ is a coherent $\Orb_{X_{\cl}}$-module. Then
  $\alpha_{\cl}^*\alpha_{\cl,*}\mathcal{H}^0(M) \to \mathcal{H}^0(M)$
  is surjective
  \cite[\href{https://stacks.math.columbia.edu/tag/0891}{Tag
    0891}]{stax} and so
  $\mathcal{H}^0(\LDER \alpha^*i_*\alpha_{\cl,*}\mathcal{H}^0(M)) \to
  \mathcal{H}^0(M)$ is surjective, where $i\colon X_{0,\cl} \hookrightarrow X_0$ is the inclusion. By \cite[\href{https://stacks.math.columbia.edu/tag/0GRF}{Tag
    0GRF}]{stax}, we can write $\alpha_{\cl,*}\mathcal{H}^0(M)=\cup_{\gamma\in \Gamma} M_\gamma$ as an increasing union of coherent $\Orb_{X_{0,\cl}}$-submodules. Since $\mathcal{H}^0(M)$ is a coherent $\Orb_{X_{\cl}}$-module, for $\gamma$ sufficiently large, the induced maps $\alpha_{\cl}^*M_\gamma \to \alpha_{\cl}^*\alpha_{\cl,*}\mathcal{H}^0(M) \to \mathcal{H}^0(M)$ are all surjective. Hence, there is an $M_0 \in \Ds_{\coh}^-(X_0)^\heartsuit$, where $\mathcal{H}^0(M_0)$ is a 
  coherent $\Orb_{X_{0,\cl}}$-submodule of $\alpha_{\cl,*}\mathcal{H}^0(M)$, together
  with a map $\mu \colon \LDER \alpha^*M_0 \to M$ such that
  $\mathcal{H}^0(\mu)$ is surjective. Form the fiber sequence
  $K \to \LDER \alpha^*M_0 \to M$; then we obtain an exact
  sequence of $\Orb_{S_{\cl}}$-modules:
  \[
    \mathcal{H}^{-t-1}(\RDER f_*(P \ltensor_{\Orb_X} \LDER \alpha^*M_0)) \to \mathcal{H}^{-t-1}(\RDER f_*(P \ltensor_{\Orb_X} M)) \to \mathcal{H}^{-t}(\RDER f_*(P \ltensor_{\Orb_X} K)).
  \]
  By Proposition
  \ref{prop:relative-compactness-properties}\ref{prop:relative-compactness-properties:base-change},
  $\RDER f_*(P \ltensor_{\Orb_X} \LDER \alpha^*M_0) \simeq \LDER
  a^*\RDER f_{0,*}(P_0 \ltensor_{\Orb_{X_0}} M_0)$. By Lemma
  \ref{lem:compact-pushforward},
  $\LDER a^*\RDER f_{0,*}(P_0 \ltensor_{\Orb_{X_0}} M_0) \in
  \Ds_{\coh}^-(S)$. Then the inductive hypothesis implies that
  $\tau^{\geq -t}\RDER f_*(P \ltensor_{\Orb_X} K) \in
  \Ds_{\coh}^-(S)$.  Since ${S_{\cl}}$ is noetherian,
  $\mathcal{H}^{-t-1}(\RDER f_*(P \ltensor_{\Orb_X} M)) \in
  \coh(S_{\cl})$. By the inductive hypothesis, it follows that
  $\tau^{\geq -t-1}(\RDER f_*(P \ltensor_{\Orb_X} M)) \in
  \Ds_{\coh}^-(S)$. This proves
  \ref{thm:noetherian-base-changes-derived:qaff-noeth}.
  
  For \ref{thm:noetherian-base-changes-derived:qaff-diag}: this is
  smooth-local on $S$. Hence, we may assume that $S=\Spec A$ and there
  is a factorization $\Spec A \to \Spec A_0 \to S_0$, where $\Spec A_0 \to S_0$ is smooth. Since $S_0$ has quasi-affine diagonal, all these maps are quasi-affine. It
  suffices to prove that if $M$ is connective and pseudo-coherent on
  $X$, then $\RDER f_*(P \ltensor_{\Orb_X} M)$ is pseudo-coherent on
  $S$. Let $t\geq 0$. Then $\tau^{\geq -t-d}M$ is finitely
  $(t+d)$-presented \cite[\S4.5.2]{SAG}. Now write $\pi_0(A)$ as a union of finitely generated
  $\pi_0(A_0)$-subalgebras. Since $A_0 \to A$ factors through $A_0 \to \pi_0(A_0)$, we have expressed $A$ as a filtered colimit of
  almost perfect $A_0$-algebras $\{A_\lambda\}_{\lambda\in
    \Lambda}$. Let $X_{A_0} = X_0 \times^{\RDER}_{S_0} \Spec A_0$. Then there exist a finitely $(t+d)$-presented
  $M_\lambda$ on
  $X_\lambda = X_0 \times^{\RDER}_{S_0} \Spec A_\lambda\simeq X_{A_0} \times^{\RDER}_{\Spec A_0} \Spec A_\lambda$ and a
  $\Orb_X$-equivalence
  $\tau^{\geq -t-d}M \simeq \tau^{\geq -t-d}\LDER
  \alpha_\lambda^*M_\lambda$, where
  $\alpha_\lambda \colon X \to X_\lambda$ is the projection (Theorem
  \ref{thm:fp-approx}). Let
  $f_\lambda \colon X_\lambda \to \Spec A_\lambda$,
  $\beta_\lambda \colon X_\lambda \to X_0$, and
  $a_\lambda \colon \Spec A \to \Spec A_\lambda$ be the induced
  maps. Set $P_\lambda = \LDER \beta_\lambda^*P_0$; then
  $\LDER \alpha_\lambda^*P_\lambda \simeq P$.  Since $P$ is compact in
  $\Dsqc(X)$, it follows from Lemma
  \ref{lem:strongproj}\ref{lem:strongproj:bound} that there are
  equivalences:
  \begin{align*}
    \tau^{\geq -t} \RDER f_*(P \ltensor_{\Orb_X} M)
    &\simeq  \tau^{\geq -t} \RDER f_*(P \ltensor_{\Orb_X} \tau^{\geq -t-d}M) \simeq \tau^{\geq -t} \RDER f_*(P \ltensor_{\Orb_X} \tau^{\geq -t-d}\LDER \alpha^*_\lambda M_\lambda)\\
    &\simeq \tau^{\geq -t} \RDER f_*\LDER \alpha^*_\lambda(P_\lambda \ltensor_{\Orb_{X_\lambda}}  M_\lambda)\\
    &\simeq \tau^{\geq -t} \LDER a_\lambda^*\RDER f_{\lambda,*}(P_\lambda \ltensor_{\Orb_{X_\lambda}}  M_\lambda),
  \end{align*}
  But $A_\lambda$ is noetherian, $M_\lambda$ is pseudo-coherent (Example \ref{ex:fp-noeth}), and $\Spec A_\lambda \to S_0$ is quasi-affine, so
  $\RDER f_{\lambda,*}(P_\lambda \ltensor_{\Orb_{X_\lambda}}
  M_\lambda)$ is pseudo-coherent by \ref{thm:noetherian-base-changes-derived:qaff-noeth}. Thus,
  $\tau^{\geq -t}\RDER f_*(P \ltensor_{\Orb_X} M)$ is $t$-perfect. As
  $t\geq 0$ was arbitrary, $\RDER f_*(P \ltensor_{\Orb_X} M)$ is
  pseudo-coherent (Example \ref{ex:truncation-perfect}).
\end{proof}
\begin{remark}\label{rem:kiehl-r-perfect}
  The proof of Theorem \ref{thm:noetherian-base-changes-derived} shows
  that if $M \in \Dsqc(X)$ is $r+d(P_0,f_0)$-perfect relative to $S$,
  then $\RDER f_*(\LDER \alpha^*P_0 \ltensor_{\Orb_X} M)$ is
  $r$-perfect relative to $S$. This is a derived analog of
  \cite[Theorem 2.9]{Kiehl1972}.
\end{remark}
\section{Underiving}
We now apply our derived results to the main results of the article.
\begin{proof}[Proof of Theorem \ref{thm:mainprime}]
  Let $\tilde{X} = X_0 \times^{\RDER}_{S_0} S$ and
  $\tilde{Y} = Y_0 \times^{\RDER}_{S_0} S$; then the maps
  $\tilde{f} \colon \tilde{X} \to \tilde{Y}$, $\tilde{Y}\to S$ (and
  hence their composition $\tilde{X} \to S$) are all almost perfect
  (Example \ref{ex:noetherian-ap-alg-rel}). Let
  $\tilde{\alpha} \colon \tilde{X} \to X_0$, $i_X \colon X \to \tilde{X}$ and $i_Y \colon Y \to \tilde{Y}$ be the induced
  morphisms. Then $\tilde{X}_\cl \simeq X$, $\tilde{Y}_\cl \simeq Y$,
  $\tilde{f}_\cl \simeq f$, and $\tilde{\alpha}_{\cl} \simeq
  \alpha$. Set $P=\LDER \alpha^*P_0$; then $\tilde{P}=\LDER \tilde{\alpha}^*P_0$ is $\tilde{Y}$-compact (Proposition \ref{prop:relative-compactness-properties}\ref{prop:relative-compactness-properties:base-change}) and $\LDER i_X^*\tilde{P} \simeq P$. If $M\in \Dsqc(X)$ is $S$-pseudo-coherent, then $\RDER i_{X,*}M$ is $S$-pseudo-coherent (Example \ref{ex:pushforward-closed-rel-perfect}). Then Example \ref{ex:affine-projection-formula} implies that
  \[
    \RDER \tilde{f}_*(\tilde{P} \ltensor_{\Orb_{\tilde{X}}} \RDER i_{X,*}M) \simeq \RDER \tilde{f}_*\RDER i_{X,*}(P \ltensor_{\Orb_X} M) \simeq \RDER i_{Y,*} \RDER f_*(P \ltensor_{\Orb_X}M).
  \]
  Thus, by Example \ref{ex:pushforward-closed-rel-perfect} and Proposition \ref{proposition: rel-abs-agree}, it
  suffices to show
  $\RDER \tilde{f}_*(\tilde{P}
  \ltensor_{\Orb_{\tilde{X}}} - )\colon \Dsqc(\tilde{X}) \to
  \Dsqc(\tilde{Y})$ preserves pseudo-coherent objects, which is just
  Theorem
  \ref{thm:noetherian-base-changes-derived}\ref{thm:noetherian-base-changes-derived:qaff-diag}.
\end{proof}
\begin{remark}\label{rem:kiehl-r-perfect-underived}
  Remark \ref{rem:kiehl-r-perfect} implies that if $M \in \Dqc(X)$ is
  $r+d(P_0,f_0)$-pseudo-coherent relative to $S$, then
  $\RDER f_*(\LDER \alpha^*P_0 \ltensor_{\Orb_X} M)$ is
  $r$-pseudo-coherent relative to $S$.
\end{remark}
\begin{proof}[Proof of Theorem \ref{thm:gms}]
      By Proposition \ref{prop:relative-compactness-properties}, we
  may assume $Y=\Spec B$ and
  $S=\Spec A$.  Since $B$ is a
  finitely generated $A$-algebra, there is a closed immersion of
  $S$-schemes $i \colon Y \hookrightarrow Y'$, where $Y'$ is affine
  and finitely presented over $S$ (e.g., can take $Y'= \A^m_S$ for
  some $m\gg 0$). Note that an $\Orb_Y$-module $M$ is pseudo-coherent
  relative to $S$ if and only if the $\Orb_{Y'}$-module $\RDER i_*M$
  is pseudo-coherent relative to $S$ (Example
  \ref{E:relative-pseudo-push}). Hence, we may replace $f$ with
  $i\circ f$ and assume that $Y \to S$ is finitely presented.
  
  By \cite[Corollary C \& Theorem
  5.10]{rydh2023absolutenoetherianapproximationalgebraic}, there is a
  closed immersion $j \colon X \hookrightarrow X'$, where
  $f'\colon X' \to Y$ is of finite presentation with affine diagonal
  and admits a proper good moduli space
  $f_{\gms}' \colon X_{\gms}' \to Y$. Then
  $\RDER f_* \simeq \RDER f'_* \RDER j_*$ and $\RDER j_*$ preserves
  $S$-pseudo-coherence (Example \ref{E:relative-pseudo-push}). Hence,
  we may now replace $X$ and $f$ with $X'$ and $f'$, so that $f$ is
  finitely presented. We now do relative noetherian approximation
  (e.g., combine \cite{Rydh_2015} with \cite[Corollary
  7.5]{MR5068605}): this produces a noetherian subalgebra
  $A_0 \subseteq A$, a finite type $A_0$-algebra $B_0$ with an
  $A_0$-algebra map $B_0 \to B$ such that
  $B_0 \otimes_{A_0} A \simeq B$, and a finite type morphism with
  affine diagonal $f_0 \colon X_0 \to \Spec B_0$ that admits a proper
  good moduli space such that $f_0 \otimes_{A_0} A \simeq
  f$. Combining \cite[Theorem 6.3.3]{alper_gms} with
  \cite[\href{https://stacks.math.columbia.edu/tag/08AR}{Tag
    08AR}]{stax}, we see that $f_0 \colon X_0 \to \Spec B_0$ is
  cohomologically proper. The result now follows from Theorem
  \ref{thm:mainprime}.
\end{proof}
\begin{remark}
    The proof of Theorem \ref{thm:gms} shows that if we have a factorization $X \xrightarrow{\pi} X_{\gms} \xrightarrow{f_{\gms}} Y$, where $X \to Y$ is locally of finite type and $\pi$ is a good moduli space, then $f_{\gms}$ is of finite type. This is a finite type variant of \cite[Theorem 6.1]{MR5068605}.
\end{remark}
\begin{proof}[Proof of Theorem \ref{thm:firstprime}]
  By Proposition \ref{prop:relative-compactness-properties}, we
  immediately reduce to the situation where $Y=\Spec B$ and
  $S=\Spec A$ are affine schemes and $P$ is a compact object of
  $\Dqc(X)$. Arguing as in the proof of Theorem \ref{thm:gms}, we may assume that $Y \to S$ is finitely presented and there exist a 
  closed immersion $\imath \colon X \hookrightarrow X'$ and a proper
  and finitely presented morphism with finite diagonal
  $f'\colon X' \to Y$. Let $P'$ be a compact generator of $\Dqc(X')$
  \cite[Theorem A]{perfect_stax}; then $\LDER \imath^*P'$ compactly
  generates $\Dqc(X)$ \cite[Lemma 8.2]{perfect_stax}. In particular,
  $P$ lies in the smallest thick triangulated subcategory containing
  $\LDER\imath^*P'$ (Thomason's Localization Theorem \cite[Theorem
  3.2]{perfect_stax}). Since the collection of compacts $Q$ where $\RDER f_*(Q \ltensor_{\Orb_X} - )$ preserves $S$-pseudo-coherence is thick, it suffices to prove that $
    \RDER f_*( \LDER \imath^*P' \ltensor_{\Orb_X} - ) \simeq \RDER f'_*
    \RDER \imath_*(\LDER \imath^*P' \ltensor_{\Orb_X} - ) \simeq \RDER
    f'_*(P'\ltensor_{\Orb_{X'}} \RDER \imath_*(-))$
  preserves pseudo-coherent complexes relative to $S$ for any compact generator $P'$ of $\Dqc(X')$. Since $\RDER \imath_*$ preserves $S$-pseudo-coherence (Example
  \ref{E:relative-pseudo-push}), we may replace $X$ by $X'$ and
  assume that $f$ is also finitely presented.

  We now do relative noetherian approximation \cite{Rydh_2015}: this
  produces a finite type $\Z$-subalgebra $A_0 \subseteq A$, a finite
  type $A_0$-algebra $B_0$ with an $A_0$-algebra map $B_0 \to B$ such
  that $B_0 \otimes_{A_0} A \simeq B$, and a proper morphism with
  finite diagonal $f_0 \colon X_0 \to \Spec B_0$ such that
  $f_0 \otimes_{A_0} A \simeq f$. If $Q_0$ is a compact generator of $\Dqc(X_0)$, then
   $Q=\LDER \alpha^*Q_0$ compactly generates $\Dqc(X)$, where
  $\alpha \colon X \to X_0$ is the induced projection \cite[Lemma 8.2]{perfect_stax}. Now combine Theorem
  \ref{thm:mainprime} with \cite[Theorem 1.2]{proper_coverings}.
\end{proof}
\begin{lemma}\label{lem:commutative-algebra-conditions}
  Let $f\colon X\to Y$ be a quasi-compact and quasi-separated morphism
  of algebraic stacks, where $X$ and $Y=\Spec B$ are locally of finite
  type over an affine scheme $S=\Spec A$. Let $P \in \Dqc(X)$ be
  $Y$-compact. If $\RDER f_*(P \ltensor_{\Orb_X} -)$ preserves
  pseudo-coherent complexes relative to $S$, then it preserves
  complexes that are
\begin{enumerate}[label = \normalfont(\roman*)]
\item \label{lem:commutative-algebra-conditions:pcpr} pseudo-coherent and perfect relative to $S$;
\item \label{lem:commutative-algebra-conditions:pc} pseudo-coherent if $f$ is pseudo-coherent and $Y=S$;
\item \label{lem:commutative-algebra-conditions:perf} perfect if $f$ is perfect and $Y=S$.
\end{enumerate}
\end{lemma}
\begin{proof}
    For \ref{lem:commutative-algebra-conditions:pcpr}: let $M$ be pseudo-coherent and perfect
  relative to $S$. Since 
  $\RDER f_*(P \ltensor_{\Orb_X} M)$ is pseudo-coherent relative to
  $S$, it suffices to prove that it is perfect relative to $S$. Let $N \in \Mod(A)$; then Proposition \ref{prop:relative-compactness-properties}\ref{prop:relative-compactness-properties:projection} implies that 
  \begin{align*}
    \RDER f_*(P \ltensor_{\Orb_X} M) \ltensor_A N &\simeq   \RDER f_*(P \ltensor_{\Orb_X} M) \ltensor_B (B\ltensor_A N)\\
                                                &\simeq \RDER f_*(P \ltensor_{\Orb_X} M \ltensor_{\Orb_X} \LDER f^*(B \ltensor_A N))\\
    &\simeq \RDER f_*(P \ltensor_{\Orb_X} (M \ltensor_{A} N)).
  \end{align*}
  Since $M$ is perfect relative to $S$, $M \ltensor_{A} N$ has
  amplitude contained in $[-a,a]$ for some $a\geq 0$ (independent of
  $N$). It follows that
  $\RDER f_*(P \ltensor_{\Orb_X} (M \ltensor_{A} N))$ has amplitude
  contained in some $[-a-r,a+r]$ for some $r$ independent of $N$. That
  is, $ \RDER f_*(P \ltensor_{\Orb_X} M)$ has finite tor-amplitude
  relative to $A$, so is perfect relative to $S$.

  For \ref{lem:commutative-algebra-conditions:pc}: since $f$ is
  pseudo-coherent, a complex on $X$ is pseudo-coherent relative to $S$
  if and only if it is pseudo-coherent (Example
  \ref{E:noetherian-smooth-pcoh-relative}). Hence, if $M \in \Dqc(X)$
  is pseudo-coherent, it is pseudo-coherent relative to $S$. Thus,
  $\RDER f_*(P \ltensor_{\Orb_X} M)$ is pseudo-coherent relative to
  $S=Y$, which gives the claim. For \ref{lem:commutative-algebra-conditions:perf}: combine \ref{lem:commutative-algebra-conditions:pcpr} and
  \ref{lem:commutative-algebra-conditions:pc}.  
\end{proof}
\begin{proof}[Proof of Corollary \ref{cor:tame-perfect}]
    By Theorem \ref{thm:tame-firstprime} with $Y=S$, $\RDER f_*$ preserves $S$-pseudo-coherent complexes. Now apply Lemma \ref{lem:commutative-algebra-conditions}\ref{lem:commutative-algebra-conditions:perf}.
\end{proof}
\appendix
\section{Flatness and finiteness}\label{APP:flatness}
Recall that the
$\infty$-category of \emph{animated rings} $\AniRing$ is the $\infty$-category
freely generated under sifted colimits by finite polynomial rings over
$\Z$. This is equivalent to the localisation of the category of
simplicial rings at the weak equivalences. For
background material, the reader is referred to \cite[\S5]{purity_flat_cohoom},
\cite[\S2.7 \& \S25]{SAG}, and \cite[\S7.2.4]{Lurie-HA}.
\begin{example}\label{E:ring-ani}
    If $\Ring$ denotes the category of rings, then there is a fully faithful embedding $i\colon \Ring \to \AniRing$. The functor $\pi_0 \colon \AniRing \to \Ring$ is left adjoint to $i$.
\end{example}
Let $A$ be an animated ring. Let $\Ds(A)$ be the stable
$\infty$-category of $A$-modules. Note that there is an equivalence between connective $A$-modules
(i.e., those $M$ with $\H^{i}(M)=\pi_{-i}(M) = 0$ for all $i>0$) and animated
$A$-modules. There also is a standard $t$-structure on $\Ds(A)$, with
$\Ds(A)^{\leq 0}=\Ds(A)_{\geq 0}$ consisting of connective $A$-modules, and an
equivalence of abelian categories
$\Ds(A)^{\heartsuit} \simeq \Mod(\pi_0(A))$.  If $A$ is
\emph{discrete} (i.e., $A \simeq \pi_0(A)$), then there is a
triangulated equivalence between the homotopy category
$\mathrm{Ho}(\Ds(A))$ and $\D(\pi_0(A))$.

We now briefly recall some results on flatness in the animated
setting. Let $A$ be an animated ring. An $A$-module $M$ is \emph{flat}
if $\H^0(M)$ is a flat $\pi_0(A)$-module and the induced map
$\pi_n(A) \otimes_{\pi_0(A)} \H^0(M) \to \H^{-n}(M)$ is an isomorphism
of abelian groups for all $n\in \Z$
\cite[Definition 7.2.2.10]{Lurie-HA}. In particular, since $A$ is
connective, flat $A$-modules are always connective. A flat $A$-module
$M$ is \emph{faithfully} flat if $M \ltensor_A N \simeq 0$ implies that
$N\simeq 0$ for any $A$-module $N$. We say that an $A$-algebra $C$ is
\emph{(faithfully) flat} if its underlying $A$-module is so. We also
say that a flat and almost finitely presented $A$-algebra $C$ is
\emph{smooth} or \emph{\'etale} if the induced map
$\pi_0(A) \to \pi_0(C)$ is a smooth or \'etale map of classical rings,
respectively.
\begin{example}\label{E:faithfully-flat-discrete}
  Let $A$ be a discrete animated ring. An $A$-module $M$ is flat if
  and only if $\H^i(M) = 0$ for all $i\neq 0$ and $\H^0(M)$ is a flat
  $\pi_0(A)$-module.
\end{example}
\begin{example}\label{E:faithfully-flat-derived}
  Let $A$ be an animated ring. An $A$-module $M$ is faithfully flat if
  and only if $M$ is a flat $A$-module and
  $\H^0(M)\otimes_{\pi_0(A)} \kappa(\mathfrak{p}) \neq 0$ for all
  prime ideals $\mathfrak{p} \subseteq \pi_0(A)$. The condition is
  certainly necessary. For the sufficiency, it suffices to prove
  $M \ltensor_A N \simeq 0$ implies that $N\simeq 0$. Since $M$ is a
  flat $A$-module, it follows that
  $\H^n(M \ltensor_A N) \simeq \H^0(M) \otimes_{\pi_0(A)} \H^n(N)$ for
  all $n\in \Z$ \cite[Proposition 7.2.2.13]{Lurie-HA}. But the residue field
  condition on $\H^0(M)$ implies that it is a faithfully flat
  $\pi_0(A)$-module
  \cite[\href{https://stacks.math.columbia.edu/tag/00HP}{Tag
    00HP}]{stax}, so $\H^n(N) = 0$ for all $n\in \Z$. Hence,
  $N \simeq 0$ and $M$ is a faithfully flat $A$-module.
\end{example}
\begin{example}\label{E:faithfully-flat-algebra}
  A flat $A$-algebra $C$ is faithfully flat if and only if the induced
  morphism $\Spec \pi_0(C) \to \Spec \pi_0(A)$ is surjective. This
  follows immediately from Example \ref{E:faithfully-flat-derived}.
\end{example}
We next show that relative perfection descends along
faithfully flat algebras.
\begin{proposition}\label{P:flat-descent-r-perfect}
  Let $A \to B$ be a map of animated rings and let $M$ be a
  $B$-module. Let $A \to C$ be a faithfully
  flat $A$-algebra. Let $r\in \Z$. If the $B\ltensor_AC$-module
  $M \ltensor_A C$ is $r$-perfect relative to $C$, then $M$ is
  $r$-perfect relative to $A$.
\end{proposition}
\begin{proof}
  Since $A \to C$ is flat,
  $\H^n(M \ltensor_A C) \simeq \H^n(M) \otimes_{\pi_0(A)} \pi_0(C)$
  \cite[Proposition 7.2.2.13]{Lurie-HA}. Since $M \ltensor_A C$ is
  $r$-perfect relative to $C$, it follows that it is almost
  connective, so $\H^n(M) \otimes_{\pi_0(A)} \pi_0(C) = 0$ for all
  $n \gg 0$.  But $A \to C$ is faithfully flat, so $\H^n(M) = 0$ for
  all $n \gg 0$ and $M$ is almost connective. By Example
  \ref{ex:relperf-nilp}, we may thus assume that $A$ and $C$ are
  discrete and $M$ is almost connective. Arguing similarly, we also
  see that $\pi_0(A) \to \pi_0(B)$ is of finite type. Now choose
  $b_1$, $\dots$, $b_n \in \pi_0(B)$ such that the induced map
  $A[x_1,\dots,x_n] \to B$ induces a surjection on $\pi_0$. Since
  $M \ltensor_A C$ is $r$-perfect relative to $C$, $M \ltensor_A C$ is
  a $r$-perfect $C[x_1,\dots,x_n]$-module. By
  \cite[\href{https://stacks.math.columbia.edu/tag/068R}{Tag
    068R}]{stax}, $M$ is an $r$-perfect $A[x_1,\dots,x_n]$-module, so $M$ is an
  $r$-perfect $B$-module relative to $A$.
\end{proof}

The remainder of this section is not needed for the results in the
paper.
\begin{example}\label{ex:derived-quotient}
  Let $A$ be an animated ring and $f_1$, $\dots$, $f_n \in
  \pi_0(A)$. Set
  $A/\!\!/(f_1,\dots,f_n) := \Z \ltensor_{\Z[\underline{X}]} A$, where
  $\underline{X} = (X_1,\dots,X_n)$, $\Z[\underline{X}] \to \Z$ sends
  the $X_i$ to $0$ and $\Z[\underline{X}] \to A$ sends $X_i$ to
  $f_i$. Since $\Z$ is a perfect $\Z[\underline{X}]$-module, 
  $A/\!\!/(f_1,\dots,f_n)$ is a perfect $A$-module and an almost
  finitely presented $A$-algebra. Also,
  $\pi_0(A/\!\!/(f_1,\dots,f_n)) \simeq
  \pi_0(A)/(f_1,\dots,f_n)$. Note that if $\phi \colon A \to B$ is
  a map of animated rings, then it is easily verified that
  there is an equivalence of $B$-algebras
  $A/\!\!/(f_1,\dots,f_n) \ltensor_A B \simeq
  B/\!\!/(\phi(f_1),\dots,\phi(f_n))$. Similarly, if $R \to A$ is a
  map of animated rings, then
  $R \ltensor_{R[\underline{X}]} A \simeq A/\!\!/(f_1,\dots,f_n)$. 
\end{example}
\begin{example}\label{ex:derived-finite-monogenic}
  Let $A$ be an animated ring. Let $f(t) \in \pi_0(A)[T]$ be a monic
  polynomial of degree $n$. Then $A[T]/\!\!/(f)$ is a free $A$-module
  of rank $n$; in particular, $A[T]/\!\!/(f)$ is a perfect $A$-module
  and an almost perfect $A$-algebra. 
  Indeed, the $T^i$ for $i=0$, $\dots$,
  $n-1$ induce an $A$-module homomorphism
  $A^{\oplus n} \xrightarrow{\psi}A[T]/\!\!/(f)$. By Example
  \ref{ex:aperf-nilp}, it suffices to prove that
  $\psi \ltensor_A \pi_0(A)$ is an equivalence. By Example \ref{ex:derived-quotient},
  \[
    A[T]/\!\!/(f) \ltensor_A \pi_0(A) \simeq     A[T]/\!\!/(f) \ltensor_{A[T]} (A[T] \ltensor_A \pi_0(A)) \simeq \pi_0(A)[T]/\!\!/(f),
  \]
  so we may assume $A$ is discrete. {In this case, it is clear that $\psi$ induces an isomorphism on $\H^0$, and thus it suffices to show that 
  $\H^i(A[T]/\!\!/ (f)) = 0$ for $i < 0$. This follows from the fact that $f$ is monic, and hence the multiplication by $f$ map $A[T] \xrightarrow{\times f} A[T]$ is injective.} 
\end{example}
The following result is similar to
\cite[\href{https://stacks.math.columbia.edu/tag/0673}{Tag
  0673}]{stax}.
\begin{corollary}\label{corollary:pushforward}
  Let $A \to B \to C$ be morphisms of animated rings such that
  $\pi_0(A) \to \pi_0(B)$ is of finite type and
  $\pi_0(B) \to \pi_0(C)$ is finite. Let $M$ be a $C$-module that is
  $r$-perfect relative to $A$; then $M$---regarded as a
  $B$-module---is $r$-perfect relative to $A$.
\end{corollary}
\begin{proof}
  Choose $b_1$, $\dots$, $b_n \in \pi_0(B)$ such that the induced map
  $P \to B$ induces a surjection
  $\pi_0(P) \to \pi_0(B)$, where
  $P=A[\underline{X}]$ and $\underline{X} = (X_1,\dots,X_n)$. Next
  choose $c_1$, $\dots$, $c_m \in \pi_0(C)$ such that the induced map
  $B[\underline{Y}] \to C$, where $\underline{Y}=(Y_1,\dots,Y_m)$,
  induces a surjection
  $\pi_0(B[\underline{Y}])\simeq \pi_0(B)[\underline{Y}] \to
  \pi_0(C)$. Since $M$ is $r$-perfect relative to $A$, it is an $r$-perfect $P[\underline{Y}]$-module. Hence, it remains to show that $M$ is an $r$-perfect $P$-module.

  The finiteness of $\pi_0(B) \to \pi_0(C)$ says that there are monic
  polynomials $f_j(T) \in \pi_0(P)[T]$ such that
  $f_j(c_j) =0$ in $\pi_0(C)$ for $j=1$, $\dots$, $m$. Let $P_0' = P$ and if $j>0$ let $P_j' = P[Y_1,\dots, Y_j]$. Also let 
  $C'_0 = P'_0$ and if $j>0$ let
  $C'_j=P_j'/\!\!/(f_1(Y_1),\dots,f_j(Y_j))$. Then for all $j\geq 0$, 
  $C'_j$ is perfect as a
  $P_j'$-module and almost perfect as a $P_j'$-algebra (Example
  \ref{ex:derived-quotient}). The diagram:
  \[
    \begin{tikzcd}
      \Z[\underline{S},T] \arrow[r] \arrow[d] & \Z[T] \arrow[r] \arrow[d] & \arrow[d] \Z \\
      P'_j \arrow[r] & C_{j-1}'[Y_j] \arrow[r] &   C'_j,
    \end{tikzcd}
  \]
  where $\underline{S}=(S_1,\dots,S_{j-1})$, $S_i \mapsto f_i(Y_i)$
  and $T \mapsto f_j(Y_j)$, has squares that are homotopy pushouts and
  shows that $C'_j \simeq C'_{j-1}[Y_j]/\!\!/(f_j(Y_j))$. Hence, if
  $j>0$, then $C_j'$ is a finite free $C'_{j-1}$-module (Example
  \ref{ex:derived-finite-monogenic}). By induction, it follows that
  $C'_m=P[\underline{Y}]/\!\!/(f_1(Y_1),\dots,f_m(Y_m))$ is a finite
  free $P$-module. Since $C'_m \to C$ induces a surjection
  $\pi_0(C'_m) \to \pi_0(C)$ and $C'_m$ is an almost perfect
  $A$-algebra, $M$ is an $r$-perfect $C'_m$-module (Corollary
  \ref{cor:aperf-relative-abs}). But $C'_m$ is pseudo-coherent as a
  $P$-module, so $M$ is $r$-perfect as a $P$-module \cite[Lemma
  5.6.1.2]{SAG}.
\end{proof}

\section{Derived algebraic stacks and their quasi-coherent sheaves}\label{APP:derived-stacks}
Here we fix our conventions and recall some definitions and results about
derived algebraic stacks. There are a number of more
comprehensive accounts available. A brief, but non-exhaustive, list is \cite{BZFN,MR3701352,hansenmann_cllc,khan2026lecturesalgebraicstacks,SAG}.
\subsection{Category theory}
By an \emph{$\infty$-category}, we will mean an $\infty$-category in the sense of \cite{HTT}; that is, a simplicial set satisfying the inner Horn filling conditions. These are often also called \emph{$(\infty,1)$-categories}, to emphasise that all $n$-morphisms for $n\geq 2$ are invertible. Following \cite{purity_flat_cohoom}, the \emph{$\infty$-category of} \emph{anima}, denoted $\Ani$, will refer to the $\infty$-category whose objects consist of $\infty$-groupoids (i.e., spaces or Kan complexes), viewed as $\infty$-categories in which every morphism is invertible. If $\mathcal{E}$ is an $\infty$-category and $x$ and $y$ are objects of $\mathcal{E}$, then let $\Map_{\mathcal{E}}(x,y)$ be the morphism space \cite[Definition 1.2.2.1]{HTT}.

We will let $\Cat$ and $\bCat$ denote the $\infty$-categories of small and large $\infty$-categories, respectively \cite[Definition 3.0.0.1]{HTT}. These $\infty$-categories are closed under small limits and colimits. An $\infty$-category $\mathcal{C}$ always has a \emph{core} $\mathcal{C}^{\simeq}$ \cite[Proposition 1.2.5.3]{HTT}, which is the largest subsimplicial set that is Kan. The induced functor $(\cdot)^\simeq \colon \Cat \to \Ani$ is right adjoint to the natural fully faithful functor $\Ani \to \Cat$ and so preserves limits. There is also a version of this for large $\infty$-categories.
\begin{example}\label{ex:limit-explicit}
    Let $I$ be a simplicial set and let $\mathcal{D} \colon I \to \bCat$ be a functor. If $\delta \colon j \to i$ is a map in $I$, then there is an induced functor $\mathcal{D}(\delta)^* \colon \mathcal{D}_j \to \mathcal{D}_i$. Set $\mathcal{D}_{\infty} = \varprojlim_I \mathcal{D}$. An object $d$ of $\mathcal{D}_{\infty}$ corresponds to a collection $\{d_i\}$ such that if $\delta \colon j \to i$ is a map in $I$, then there are equivalences $\mathcal{D}(\delta)^*d_j \simeq d_i$ that are compatible with $I$. More precisely, $d$ is a natural transformation from the constant functor $c\colon I \to \Cat \colon i \mapsto [0]$, where $[0]$ denotes the terminal $\infty$-category, and $\mathcal{D}$. If $e$ is another object of $\mathcal{D}_{\infty}$, then 
    \[
    \Map_{\mathcal{D}_{\infty}}(d,e) \simeq \varprojlim_I \Map_{\mathcal{D}_i}(d_i,e_i).
    \]
  \end{example}
  \begin{example}\label{ex:n-category-truncation}
    Let $\Delta$ be the simplex category. Let
    $\Delta_s \subseteq \Delta$ be the wide subcategory (i.e., the
    same objects, but we only allow those morphisms that are
    injective) \cite[Notation 6.5.3.6]{HTT}. Let $n\geq 1$ and let
    $\tCat_{\leq n} \subseteq \bCat$ be the full subcategory spanned
    by $n$-categories (equivalently, those $\infty$-categories whose
    objects are all $(n-1)$-truncated). Note that $\tCat_{\leq n}$ is
    naturally equivalent to an $(n+1)$-category \cite[Proposition
    2.3.4.18]{HTT}. Let $\Delta_{s,\leq n} \subseteq \Delta_s$ be the
    full subcategory spanned by simplices of degree $\leq n$; then
    $\Delta_{s,\leq n}$ is a finite category. Consider a functor
    $\mathcal{D} \colon \Delta_s \to \bCat$ that factors through
    $\tCat_{\leq n+1}$. Then restriction along
    $\Delta_{s,\leq n+2} \subseteq \Delta_{s}$ induces an equivalence
    $\varprojlim_{\Delta_s} \mathcal{D} \simeq
    \varprojlim_{\Delta_{s,\leq n+2}} \mathcal{D}$.
  \end{example}
Let $\LPr$ denote the (large) $\infty$-category of presentable $\infty$-categories, where we only consider those functors $g^* \colon \mathcal{C} \to \mathcal{D}$ that preserve small colimits. In particular, they admit right adjoints $g_* \colon \mathcal{D} \to \mathcal{C}$. By \cite[Proposition 5.5.3.13]{HTT}, $\LPr$ admits all small limits and the forgetful functor $\LPr \to \bCat$ preserves them. The following simple lemma (cf.\ \cite[Corollary 4.7.5.18]{HTT}), inheriting some of the notation from Example \ref{ex:limit-explicit}, will be very useful to us when we prove base change results (cf.~\cite[Proposition 9.1.5.7]{SAG}). 
\begin{lemma}\label{lem:bc-prl-limits}
   Let $I$ be a simplicial set. Consider a commuting diagram of functors $I \to \LPr$
   \[
   \begin{tikzcd}  
         \mathcal{D}' & \mathcal{D} \arrow[l,"\gamma^*"'] \\
       \mathcal{C}' \arrow[u,"f'^*"] & \arrow[l,"g^*"] \mathcal{C}\mathrlap{.} \arrow[u,"f^*"']
    \end{tikzcd}
   \]
   Denote the limits of all this data with a subscript $\infty$. 
   Assume that for all maps $\delta \colon j \to i$ in $I$, the Beck--Chevalley transformations $\mathcal{C}(\delta)^*g_{j,*} \to g_{i,*}\mathcal{C}'(\delta)^*$, $\mathcal{D}(\delta)^*\gamma_{j,*} \to \gamma_{i,*}\mathcal{D}'(\delta)^*$, and $f_i^*g_{i,*} \to \gamma_{i,*}f_i'^*$ are equivalences.
    \begin{enumerate}[label = \normalfont(\roman*)]
        \item \label{lem:bc-prl-limits:push} If $c'=\{c_i'\}$ is an object of $\mathcal{C}'_{\infty}$, then $g_{\infty,*}c' \simeq \{g_{i,*}c'_i\}$.
        \item \label{lem:bc-prl-limits:beck} 
    The Beck--Chevalley transformation $f_{\infty}^*g_{\infty,*} \to \gamma_{\infty,*}f_{\infty}'^*$ is an equivalence.
    \end{enumerate}
\end{lemma}
\begin{proof}
    Claim \ref{lem:bc-prl-limits:push} follows immediately from the explicit description of limits in Example \ref{ex:limit-explicit}.  Part \ref{lem:bc-prl-limits:beck} follows from \ref{lem:bc-prl-limits:push} and the Beck--Chevalley equivalences.
\end{proof}
\subsection{Derived algebraic stacks}
For our definitions and conventions in derived algebraic geometry, we will essentially follow \cite{khan2026lecturesalgebraicstacks}. Let $\Aff$ be the opposite of the $\infty$-category of $\AniRing$. We refer to objects of $\Aff$ as \emph{affine (derived) schemes} and (by abuse of notation) denote the affine derived scheme
associated to an animated ring $A$ by $\Spec A$. A \emph{(derived) prestack} is just a presheaf on $\Aff$. 

We now topologize  $\Aff$. The \'etale (resp.~Zariski) covering families are those finite families
$\{\Spec A_i \to \Spec A\}_{i=1}^r$, where $A \to A_i$ is \'etale (resp.~there is an $A$-algebra equivalence $A_i\simeq A[f_i^{-1}]$ for some $f_i \in \pi_0(A)$) such that $A \to \prod_{i=1}^r A_i$ is faithfully flat. It follows that we can then consider the $\infty$-category $\STK=\SHV(\Aff)$ of sheaves with respect to the \'etale topology, which we call \emph{(derived)
  stacks}. By Yoneda, there is an induced fully faithful embedding
$\Aff \to \STK$.

A \emph{derived scheme} is a derived stack that is Zariski-locally affine. A \emph{derived algebraic space} is a derived stack with diagonal that is representable by monomorphisms of derived schemes and which admits an \'etale cover by a derived scheme. Finally, a \emph{derived algebraic stack} is a derived stack with diagonal that is representable by derived algebraic spaces and which admits a smooth cover by a derived algebraic space. That is, our derived algebraic stacks are the ``derived $1$-Artin stacks'' of \cite[Definition 2.37]{khan2026lecturesalgebraicstacks}. 

Given a prestack $X \colon \AniRing \to \Ani$, its \emph{classical truncation} $X_{\cl}$ is the composition $\Ring \xrightarrow{i} \AniRing \xrightarrow{X} \Ani$ (Example \ref{E:ring-ani}). If $X$ is a derived scheme, algebraic space, or algebraic stack, then $X_\cl$ is a scheme, algebraic space, or algebraic stack in the sense of \cite{stax}. 
Conversely, given a pseudo-functor $Y \colon \Ring \to \mathsf{Grpd} $, there is an associated derived algebraic stack $i(Y)$, which is the right Kan extension of $Y \colon \Ring \to \mathsf{Grpd} \to \Ani$ along $i \colon \Ring \to \AniRing$, where $\mathsf{Grpd}$ denotes the $2$-category of groupoids. If $Y$ is an algebraic stack, then $i(Y)$ is a derived algebraic stack. These assignments are functorial and the functor from classical algebraic stacks to derived algebraic stacks is fully faithful. 

A derived algebraic stack is \emph{locally noetherian} if it is smooth-locally $\Spec A$, where $A$ is a derived noetherian ring. As expected, properties of maps of animated rings that are local on the source and target for smooth coverings can be transported to morphisms of derived algebraic stacks (e.g., almost finitely presented, smooth, flat). Purely topological properties (e.g., quasi-compactness) for morphisms of classical algebraic stacks transfer naturally to morphisms of derived algebraic stacks by considering the underlying classical morphisms. For example, a morphism of derived algebraic stacks $f \colon X \to S$ is \emph{locally almost of finite presentation} if for every commuting diagram:
\[
\begin{tikzcd}
    \Spec B\arrow[r] \arrow[d] & X \arrow[d, "f"] \\
\Spec A \arrow[r]                & S,
\end{tikzcd}
\]
where $\Spec B \to X$ and $\Spec A \to S$ are smooth, $B$ is an almost finitely presented $A$-algebra. We say that $f$ is \emph{almost finitely presented} if it is quasi-compact, quasi-separated, and locally almost of finite presentation.
\subsection{Quasi-coherent sheaves}
If $X$ is a derived stack, then we define the category of quasi-coherent sheaves on $X$ as the limit (taken in $\infty$-categories):
\[
\Dsqc(X) := \varprojlim_{\Spec A \to X} \Ds(A).
\]
Since the $\Ds(A)$ are all stable and the maps in the defining system are functors of stable $\infty$-categories, it follows that $\Dsqc(X)$ is stable.  
\begin{example}\label{ex:affine-mod}
    If $R$ is an animated ring and $X=\Spec R$, then $\Dsqc(X) \simeq \Ds(R)$. More generally, if $\{R_i\}_{i\in I}$ is a set of animated rings, then $\Dsqc(\amalg_i \Spec R_i) \simeq \prod_i \Ds(R_i)$. 
\end{example}
If $X$ is a derived algebraic stack, then there is an equivalence:
\[
\Dsqc(X) \simeq \varprojlim_{\Spec A \xrightarrow{\textnormal{smooth}} X} \Ds(A).
\]
In addition, if $u \colon U_0 \to X$ is a smooth cover of $X$ by a derived algebraic space, then 
\[
\Dsqc(X) \simeq \varprojlim_{[k] \in \Delta_s} \Dsqc(U_k),
\]
where $U_\bullet$ denotes the \v{C}ech nerve of $u$; in particular, each $U_k=U_0 \times_X \cdots \times_X U_0$ is a derived algebraic space. 
Bootstrapping this to derived algebraic spaces and using covers by disjoint unions of affine derived schemes, then to derived algebraic spaces with affine diagonal, we can use Example \ref{ex:affine-mod} to write:
\begin{equation}
    \Dsqc(X) \simeq \varprojlim_{i\in I} \Ds(A_i), \label{eq:small-limit-dsqc}    
\end{equation}
where $I$ is a small $1$-category and there are smooth maps $\Spec A_i \to X$ such that the transition maps $\Spec A_j \to \Spec A_i$ are smooth. In particular, $\Dsqc(X)$ is presentable \cite[Proposition 5.5.3.13]{HTT}.  We will freely use that $\Dsqc(X)$ also admits a $\ltensor$-structure that is compatible with colimits. 
\begin{example}\label{ex:classical-qcoh}
  If $X$ is a classical algebraic stack then there is a triangulated equivalence $\mathrm{Ho}(\Dsqc(iX)) \simeq \Dqc(X)$ \cite[Proposition 1.3]{perfect_stax}.
\end{example}
If $X$ is a derived algebraic stack, then there is also the \emph{standard $t$-structure} on $\Dsqc(X)$ with aisle $\Dsqc^{\leq 0}(X)$ spanned by the connective objects of $\Dsqc(X)$. For background on $t$-structures, see \cite[\S1.2.1, \S1.4.2]{Lurie-HA}. Note that if $\Spec A \to X$ is smooth, then the restriction $\Dsqc(X) \to \Ds(A)$ is $t$-exact.
\begin{remark}\label{rem:heart-derived}
    There is a natural equivalence of abelian categories between the heart of the standard $t$-structure $\Dsqc(X)^\heartsuit$ and $\qc(X_{\cl})$ (cf.~\cite[\S2.2.6]{SAG}).
\end{remark} 
The following is well-known \cite[Proposition 2.2.5.4, Proposition 9.1.3.1, Corollary 9.1.3.2, and Remark 9.1.3.3]{SAG} (cf.~\cite[Corollary 3.1.5.6]{MR3701352}).
\begin{lemma}\label{L:right-and-left-completeness}
    If $X$ is a derived algebraic stack, then the standard $t$-structure on $\Dsqc(X)$ is accessible. Moreover, $\Dsqc(X)$ is both left and right complete with respect to it.
\end{lemma}
\begin{proof}
    By \eqref{eq:small-limit-dsqc}, $\Dsqc(X) \simeq \varprojlim_{i\in I} \Ds(A_i)$, where the limit is over some smooth maps $\Spec A_i \to X$ with smooth transition maps. In particular, for each $n\in \Z$, it follows that for $*\in \{ >n, <n, \geq n, \leq n\}$ this equivalence restricts to:
    \begin{equation}
        \Dsqc^{*}(X) \simeq \varprojlim_{i\in I} \Ds^*(A_i). \label{eq:small-limit-dsqc-trunc}
    \end{equation}
    Hence,  $\Dsqc^{\leq 0}(X)$ is presentable \cite[Proposition 5.5.3.13]{HTT} and the $t$-structure is accessible \cite[Proposition 1.4.4.13]{Lurie-HA}. For left completeness:  there is a commuting square:
    \[
    \begin{tikzcd}
\Dsqc(X) \arrow[r] \arrow[d, "\sim" sloped] & \varprojlim_n \Dsqc^{\geq -n}(X) \arrow[d, "\sim" sloped] \\
\varprojlim_{i\in I} \Ds(A_i) \arrow[r]                                                & \varprojlim_n \varprojlim_{i \in I}\Ds^{\geq -n}(A_i),
\end{tikzcd}
    \]
    where the vertical arrows are equivalences by \eqref{eq:small-limit-dsqc-trunc}. Interchanging the limits in the bottom right corner, to show that the top arrow is an equivalence, it suffices to prove $\Ds(A)$ is left complete whenever $A$ is an animated ring. This is just \cite[Proposition 7.1.1.13]{Lurie-HA}. The right completeness is similar, so is omitted.
\end{proof}
If $f \colon X \to S$ is a morphism of derived algebraic stacks, then there is an adjoint pair
\[
\LDER f^* \colon \Dsqc(S) \leftrightarrows \Dsqc(X) \colon \RDER f_*.
\]
Indeed, we take $\LDER f^*$ to be restriction and it is clear that it preserves small colimits. It now follows that $\LDER f^*$ admits a right adjoint $\RDER f_*$ \cite[Corollary 5.5.2.9]{HTT}. General properties of $t$-structures and adjoint functors show that $\LDER f^*$ is right $t$-exact and $\RDER f_*$ is left $t$-exact with respect to the standard $t$-structure. Since $\RDER f_*$ is constructed as an adjoint functor, it may be poorly behaved (e.g., not compatible with base change). This will be addressed later by \emph{concentrated} morphisms (Definition \ref{D:concentrated}).
\begin{example}\label{ex:affine-mod-adjoints}
  If $\phi \colon A \to B$ is a map of animated rings, then there is
  an induced morphism of affine derived schemes
  $f\colon \Spec B \to \Spec A$. It follows that $\LDER f^*$ under the
  equivalences of Example \ref{ex:affine-mod} is just $- \ltensor_A
  B$. Similarly, $\RDER f_*$ is just the forgetful functor from
  $B$-modules to $A$-modules. 
\end{example}  
\begin{example}\label{ex:classical-qcoh-adjoints}
    Let $f \colon X \to S$ be a morphism of classical algebraic stacks. Then there is a commuting diagram:
    \[
    \begin{tikzcd}
\mathrm{Ho}(\Dsqc(iS)) \arrow[r, "\mathrm{Ho}(\LDER f^*)"] \arrow[d, "\sim" sloped] & \mathrm{Ho}(\Dsqc(iX)) \arrow[d, "\sim" sloped] \\
\Dqc(S) \arrow[r, "\LDER f^*_{\qc}"']                                                & \Dqc(X),
\end{tikzcd}
    \]
    where the vertical equivalences are from Example \ref{ex:classical-qcoh} and the horizontal map along the bottom is from \cite[\S1]{perfect_stax}.
    By uniqueness of adjoints, we find a similar relationship between $\mathrm{Ho}(\RDER f_*)$ and $\RDER f_{\qc,*}$, where the latter functor comes from \cite[\S1]{perfect_stax}.
\end{example}

\begin{example}\label{ex:base-change-affine}
Consider a cartesian diagram of derived algebraic stacks:
\begin{equation}
    \begin{tikzcd}
    X' \arrow[r,"\gamma"] \arrow[d,"f'"'] & X \arrow[d,"f"] \\
    S' \arrow[r,"g"] & S.
\end{tikzcd}
\label{eq:cartesian-square}
\end{equation}
If $g$ is affine, then the natural morphism of functors $\LDER f^*\RDER g_* \to \RDER \gamma_* \LDER f'^*$ is an equivalence. If $X$ and $S$
 are affine, then the result is trivial from Example \ref{ex:affine-mod-adjoints}. If $S$ is affine, then we can write $\Dsqc(X) = \varprojlim_i \Ds(B_i)$ and $\Dsqc(X') = \varprojlim_i \Ds(B_i')$, where $\Spec B'_i \simeq X'\times_X^\RDER \Spec B_i$ (note that this is over the same indexing category $I$). The affine case already considered, combined with Lemma \ref{lem:bc-prl-limits}, proves it in this case. In general, we can write $\Dsqc(S) = \varprojlim_j \Ds(A_j)$ and $\Dsqc(S') = \varprojlim_j \Ds(A_j')$, where $\Spec A_j' = S'\times_S^\RDER \Spec A_j$ (again, note that we are over the same indexing category $J$). If we set $X_j = X \times_S^\RDER \Spec A_j$ and $X'_j = X'\times_{S'}^{\RDER} \Spec A_j'$, then we have the base change result for the resulting square. It follows from Lemma \ref{lem:bc-prl-limits} again that we now have it for our original square.
 \end{example}
\begin{example}\label{ex:affine-projection-formula}
    Let $g \colon S' \to S$ be an affine morphism of derived algebraic stacks. If $M \in \Dsqc(S')$ and $N \in \Dsqc(S)$, then there is a natural equivalence:
    \[
    (\RDER g_*M) \ltensor_{\Orb_S} N \simeq \RDER g_*( M \ltensor_{\Orb_{S'}} \LDER g^*N).
    \]
    This is the affine projection formula. By Example
    \ref{ex:base-change-affine}, we may further assume
    $g$ is a map of affine derived schemes. Now
    apply Example \ref{ex:affine-mod-adjoints} (cf.~Proposition
    \ref{prop:relative-compactness-properties}).
\end{example}
\begin{example}\label{ex:bounded-below-colimits}
    Let $f\colon  X \to S$ be a quasi-compact and quasi-separated morphism of derived algebraic stacks. If $n\in \Z$, then $\RDER f_*$ restricted to $\Dsqc^{\geq n}(X)$ preserves filtered colimits. To see this, since $\Dsqc(X)$ and $\Dsqc(S)$ are right complete with respect to the standard $t$-structure, it suffices to prove that $\RDER f_*$ preserves filtered colimits when restricted to $\Dsqc^{[n,m]}(X)$. A simple induction further reduces this to the preservation of filtered colimits in $\Dsqc(X)^{\heartsuit}$. Combining Remark \ref{rem:heart-derived} with Example \ref{ex:classical-qcoh-adjoints}, we are reduced to the case where $X$ and $S$ are classical. Now apply \cite[Lemma 1.2(3)]{perfect_stax}. For $\RDER f_*$ to preserve all filtered colimits in $\Dsqc(X)$, we need $f$ to be concentrated (Definition \ref{D:concentrated} and Proposition \ref{prop:relative-compactness-properties}\ref{prop:relative-compactness-properties:sum}). It is not too difficult to argue from here that if we have a cartesian square as in \eqref{eq:cartesian-square}, where $g$ is flat, then there is an equivalence $\LDER g^*\RDER f_* \to \RDER f'_* \LDER \gamma^*$ on $\Dsqc^{+}(X)$.
  \end{example}
If $X$ is a derived algebraic stack and $n\geq 0$, we say that $M \in \Dsqc^{\leq 0}(X)$ is \emph{finitely $n$-presented} if it is $n$-truncated (i.e., $\tau^{<-n}M \simeq 0$) and $(n+1)$-perfect. It is perhaps simplest to think of finitely $n$-presented modules as being $n$-truncations of connective $(n+1)$-perfect modules \cite[Remark 2.7.1.3]{SAG}.
We let $\Dsqc^{n-fp}(X)$ denote the full subcategory spanned by those $M$ that are finitely $n$-presented. 
It will be convenient to let $\Ds^{n-fp}(A)$ denote $\Dsqc^{n-fp}(\Spec A)$.
\begin{example}\label{ex:0fp-classical}
  Let $X$ be a classical algebraic stack. Then there is a fully
  faithful embedding
  $\Dsqc^{0-fp}(iX) \subseteq \Dsqc(iX)^{\heartsuit}$. Its image under
  the equivalence $\Dsqc(iX)^{\heartsuit}\simeq \qc(X)$ is just
  $\qc^{fp}(X)$, the finitely presented $\Orb_X$-modules.
\end{example}
\begin{example}\label{ex:fp-noeth}
  Let $X$ be a noetherian derived algebraic stack. If $M$ is finitely
  $n$-presented, then $M$ is pseudo-coherent. Indeed, $M$ is
  $n$-truncated and $(n+1)$-perfect, so $M$ is almost connective and
  $\mathcal{H}^i(M)$ is a coherent $\Orb_{X_\cl}$-module for all
  $i\in \Z$ (for $i\gg 0$ and $i<-n$ it is zero and the remaining cases
  follow from Example \ref{ex:noetherian-ap}). It follows from Example
  \ref{ex:noetherian-ap} that $M$ is pseudo-coherent.
\end{example}
If $f \colon X \to S$ is a morphism of derived stacks, then $\tau^{\geq -n}\LDER f^*$ induces a natural functor from $\Dsqc^{n-fp}(S)$ to $\Dsqc^{n-fp}(X)$. We have the following useful lemma.
\begin{lemma}\label{L:nfp-finitelimit}
    Let $Y$ be a quasi-compact and quasi-separated derived algebraic stack. Let $n\geq 0$. There is a finite category $J_{n}$ and maps $\Spec B_j \to Y$ for $j\in J_{n}$ such that if $j \to j'$ in $J_{n}$, then $\Spec B_{j'} \to \Spec B_j$ is smooth and if $Y' \to Y$ is affine and $\Spec B_j' = \Spec B_j \times_Y^\RDER Y'$, then there is an equivalence of small $\infty$-categories:
    \[
    \Dsqc^{n-fp}(Y') \simeq \varprojlim_{j\in J_{n}} \Ds^{n-fp}(B_j').
    \]
\end{lemma}
\begin{proof}
Since $Y$ is quasi-compact and quasi-separated, there is a smooth cover $u_0 \colon U_0=\Spec B \to Y$ and the $(k+1)$-fold fiber products $U_k=U_0\times_Y \cdots \times_Y U_0$ are all quasi-compact and quasi-separated algebraic spaces. Base changing this along the affine morphism $Y' \to Y$, we obtain $U_0' = \Spec B'=\Spec B \times_Y Y'$ and $U_k'$ that are all affine over $U_k$. Also, if $W$ is a derived algebraic stack, then the objects of $\Dsqc^{n-fp}(W)$ are all connective and $n$-truncated, so their mapping spaces are all $n$-truncated. By \cite[Proposition 2.3.4.18]{HTT}, it follows that $\Dsqc^{n-fp}(W)$ is equivalent to an $(n+1)$-category. Hence,
    \[
    \Dsqc^{n-fp}(Y') \simeq \varprojlim_{[k] \in \Delta_s}\Dsqc^{n-fp}(U'_k) \simeq \varprojlim_{[k]\in \Delta_{s,\leq n+2}}\Dsqc^{n-fp}(U'_k).
  \]
  The second equivalence follows from Example \ref{ex:n-category-truncation}.
    Since $\Delta_{s,\leq n+2}$ is finite, we are reduced to the situation where $Y$ is a quasi-compact and quasi-separated algebraic space. Bootstrapping, we are reduced to the situation where $Y$ is a quasi-compact algebraic space with affine diagonal. Bootstrapping again, we are reduced to the situation where $Y=\Spec A$ is affine.  Finally, $\Ds^{n-fp}(A)$ is small (e.g., \cite[Remark  2.7.1.4]{SAG}).  
\end{proof}
We now have the following analog of \cite[Theorem 4.5.2.3]{SAG}. Note that Example \ref{ex:0fp-classical} shows that this recovers the ``standard limit results'' of \cite[\S8]{EGAIV3}.
\begin{theorem}\label{thm:fp-approx}
    Let $R$ be an animated ring. Consider a filtered diagram of animated $R$-algebras $\{A_\lambda\}_{\lambda \in \Lambda}$ with colimit $A_{\infty}$. Let $X \to \Spec R$ be an almost finitely presented derived algebraic stack and set $X_\lambda = X\times^{\RDER}_{\Spec R} \Spec A_\lambda$. If $n\geq 0$, then the canonical map induces an equivalence of small $\infty$-categories:
    \[
    \theta \colon \varinjlim_\lambda \Dsqc^{n-fp}(X_\lambda) \to \Dsqc^{n-fp}(X_{\infty}).
    \]
\end{theorem}
\begin{proof}
    By Lemma \ref{L:nfp-finitelimit}, there is a finite category $J_n$ and a system of smooth morphisms $\{\Spec B_j \to X\}_{j\in J_n}$ with smooth transition maps such that if $\lambda \in \Lambda \cup \{\infty\}$ and $\Spec B_{j,\lambda} = X_\lambda \times_{X}^\RDER \Spec B_j$, then there is an equivalence of small $\infty$-categories:
    \[
    \Dsqc^{n-fp}(X_{\lambda}) \simeq \varprojlim_{j\in J_n} \Ds^{n-fp}(B_{j,\lambda}).
    \]
   Hence, in the commuting diagram:
    \[
    \begin{tikzcd}
\varinjlim_{\lambda}\Dsqc^{n-fp}(X_\lambda) \arrow[r] \arrow[d, "\sim" sloped] & \Dsqc^{n-fp}(X_{\infty}) \arrow[d, "\sim" sloped] \\
\varinjlim_{\lambda}\varprojlim_{j\in J_n} \Ds^{n-fp}(B_{j,\lambda}) \arrow[r]                                                &  \varprojlim_{j \in J_n}\Ds^{n-fp}(B_{j,\infty}),
\end{tikzcd}
\]
the vertical arrows are all equivalences. Since these categories are small and $J_n$ is finite and $\Lambda$ is filtered, the limit and the colimit in the bottom left corner can be interchanged \cite[Lemma 4.5.3.1]{SAG}. It now follows from \cite[Corollary 4.5.1.10]{SAG} that the map along the bottom is an equivalence. The result follows.
\end{proof}
\bibliographystyle{dary}
\bibliography{bib}
\end{document}